\documentclass[preprint,10pt]{elsarticle} 
\usepackage{amsmath}
\usepackage{mathtools}
\usepackage{amsfonts}
\usepackage{amsthm}
\usepackage{amssymb}
\usepackage{dsfont}
\usepackage{graphicx}
\usepackage{xcolor}
\usepackage{extarrows}
\usepackage{pifont}
\usepackage{enumerate}
\usepackage[shortlabels]{enumitem}
\usepackage[english]{babel}
\usepackage[latin1]{inputenc}
\usepackage[T1]{fontenc}
\usepackage[colorlinks]{hyperref}
\usepackage[margin=0.9in]{geometry}
\usepackage{amsthm}
\everymath{\displaystyle}
\usepackage{pgf,tikz}
\usepackage{mathrsfs}
\usetikzlibrary{arrows}
\AtBeginDocument{%
\hypersetup{%
linkcolor=blue,%
citecolor=blue,%
}%
}%
\newcommand{\barint}{
         \rule[.036in]{.12in}{.009in}\kern-.16in
          \displaystyle\int  }
\def\R{{\mathbb{R}}}

\def\r{{\mathbb{R}}}
\def\N{{\mathbb{N}}}
\def\rn{{\mathbb{R}^{N}}}

\newcommand{\supp}{{\mathrm{supp}}}

\newcommand{\mconvd}{\xrightarrow[\delta\to 0]{mod}}

\def\rp{{[0,\infty)}} 
\def\ve{{\varepsilon}} 
\newcommand{\wt}{\widetilde}

\newcommand{\vp}{\varphi}
\newcommand{\bu}{{\bar{u}}}

\newcommand{\dv}{\mathrm{div}} 
\newcommand{\cA}{{\cal A}}

\newtheorem{theorem}{Theorem}
\newtheorem{definition}{Definition}[section]
\newtheorem{proposition}{Proposition}[section]
\newtheorem{lemma}{Lemma}[section]
\newtheorem{corollary}{Corollary}[section]
\newtheorem{example}{Example}[section]
\newtheorem{remark}{Remark}[section]

\begin{document}
\begin{frontmatter} 
\title{Elliptic problems with growth in  nonreflexive Orlicz spaces\\  and with
measure or $L^1$ data}  

\author{
Iwona Chlebicka \\
Institute of  Mathematics, Polish Academy of Sciences,\\ ul. \'{S}niadeckich 8, 00-656 Warsaw, Poland\\ 
e-mail: i.chlebicka@mimuw.edu.pl\\
and\\
Flavia Giannetti
 \\
Dipartimento di Matematica e Applicazioni "R.
Caccioppoli" \\ Universit\`a  degli Studi di Napoli\\
Via Cintia, 80126 Napoli,
Italia\\ 
e-mail: giannett@unina.it\\
and\\
Anna Zatorska-Goldstein \\ Institute of Applied Mathematics and Mechanics, University of Warsaw,\\ ul. Banacha 2, 02-097 Warsaw, Poland \\
e-mail: azator@mimuw.edu.pl}

\bigskip

\begin{abstract} 
We investigate solutions to  nonlinear elliptic Dirichlet problems of the type 
\[ 
\left\{\begin{array}{cl}
- \dv A(x,u,\nabla u)= \mu &\qquad \mathrm{ in}\qquad  \Omega,\\
u=0 &\qquad \mathrm{  on}\qquad \partial\Omega,
\end{array}\right.
\]  
where $\Omega$ is a bounded Lipschitz domain in $\rn$ and $A(x,z,\xi)$ is a Carath\'eodory's function. The growth of~the~monotone vector field $A$ with respect to the $(z,\xi)$ variables is expressed through some  $N$-functions  $B$ and $P$. We do not require any particular type of growth condition of such functions, so we deal with  problems in  nonreflexive spaces. When the problem involves measure data and weakly monotone operator, we prove existence. For $L^1$-data problems with strongly monotone operator we infer also uniqueness and regularity of~solutions and their gradients in the scale of Orlicz-Marcinkiewicz spaces.
\end{abstract}

\end{frontmatter}
\section{Introduction}
The main aim of this paper is to present the study of boundary value problems for a class of nonlinear elliptic equations. More precisely, we consider elliptic operators whose nonlinearity is expressed through $N$-functions which do not need to satisfy any particular growth condition. Since admitted data are merely integrable or in the space of measures, in general they do not belong to natural dual space and we do not study energy solutions but the more delicate notion of solution.

So far the effort in the research on Dirichlet problems associated to nonlinear elliptic equations   concentrates mainly on the case when modular function has a growth comparable with a polynomial or trapped between two power-type functions. This includes the well understood case when both the modular function and its conjugate satisfy the so-called $\Delta_2$  (or doubling) condition {necessary for an Orlicz space to be reflexive. Example~\ref{ex:poly-non-D2} below indicates that $\Delta_2$--condition is stronger than requirement that the growth is trapped between two power-type functions. Otherwise, i.e. when a modular function grows too slowly, too fastly, or not regularly enough, the analytical difficulties appear and significantly restrict good properties of the underlying functional space. We avoid this kind of growth restrictions and thus work in the nonreflexive space. Although this case requires an approach alternative to the classical one}, we make an attempt to convince that the basic toolkit is small and easy to handle. 

\medskip

  For the foundations of nonlinear boundary value problems in~non-reflexive Orlicz-Sobolev-type setting we refer to Donaldson~\cite{Donaldson}, Gossez~\cite{Gossez2,Gossez}, and~\cite{MT} by Mustonen and Tienari. 
In particular, in ~\cite{Donaldson}, the coefficients are assumed coercive, monotone with respect to $u$ and its derivatives, and the $N$-functions controlling their growth have conjugates satisfying the $\Delta_2$ condition. In ~\cite{Gossez2,Gossez,MT}, the authors removed  or weakened previous assumptions. Nonetheless, these research was focused on energy solutions.

\medskip

In the present paper we consider   elliptic Dirichlet problems of the type 
\begin{equation} \label{intro:ell:f}
\left\{\begin{array}{cl}
- \dv A(x,u,\nabla u)= f &\qquad \mathrm{ in}\qquad  \Omega,\\
u=0 &\qquad \mathrm{  on}\qquad \partial\Omega,
\end{array}\right.
\end{equation} 
where the monotone operator $A(x,z,\xi)$ has a growth with respect to the $(z,\xi)$ variables performed by general   $N$-functions and where the right--hand side data is merely integrable and further also in the space of measures. 

\medskip

It is worth pointing out that if  the datum $f$ only belongs to $L^1(\Omega)$ or to the set of Radon measure with finite total variation on $\Omega$, ${\cal M}(\Omega)$, a special notion of solutions has to be considered.  Indeed belonging to the duals of the natural Orlicz-Sobolev energy spaces associated with problems \eqref{intro:ell:f} is the minimal assumption on $f$ for weak solutions to be well defined.  {Our idea will be  to get a solution $u$ that is the
limit of a sequence of weak solutions to problems whose right-hand sides converge to $f$. More precisely, following~\cite{CiMa}, we choose the notion of~approximable solutions somehow  combining the notion of solutions obtained as a limit of approximation (SOLA for short) and entropy solutions, see~\cite{dall, bgSOLA,bbggpv}. When the problem involves measure data and the operator is weakly monotone, we prove existence. For $L^1$-data problems with strongly monotone operator we infer also uniqueness and regularity in the Orlicz-Marcinkiewicz spaces.

\medskip

Elliptic differential equations with the right-hand side which is less regular than naturally belonging to the dual space to the one of the leading part of the operator, have received special attention and a few main ideas of relevant notions of solution, cf.~\cite[Section~3]{IC-pocket} and references therein. The key property we expect from this special notions of solutions is uniqueness, which is not shared by distributional solutions. The classical example of Serrin~\cite{Serrin-pat} is a linear homogeneous equation of the type
$\dv (A(x)Du) = 0$ defined on a ball, with strongly elliptic and bounded, measurable matrix $A(x)$, that has at least two distributional solutions, among which only  one belongs to the natural energy space $W^{1,2}(B)$. The problem of uniqueness of very weak solutions to measure-data equations is, to our best knowledge, an~open problem. 

\medskip

 There are at least three different and already classical approaches to this kind of problems keeping uniqueness even under weak assumptions on the data. The notion of renormalized solutions  appeared first in~\cite{diperna-lions}, whereas the entropy solutions comes from~\cite{bbggpv}. The SOLA were introduced in~\cite{bgSOLA,dall}. See also~\cite{FS,DMOP} for other classical results. Under certain restrictions the mentioned notions coincide,~\cite{Rak,KiKuTu}.   Following~\cite{CiMa}, we investigate the already mentioned \textit{approximable solutions}, which for $L^1$-data  are unique.
 
  On the other hand, regularity for $L^1$ or measure data is deeply investigated in the Sobolev setting, e.g.~\cite{DV,min07,min-math-ann}, but besides little is known in general Orlicz spaces, especially outside $\Delta_2$-family, where we want to contribute. To our best knowledge, gradient estimates provided to elliptic problems posed in the Orlicz setting are restricted to \cite{ACCZG,IC-grad,CiMa}. None of this results however concerns the class of operators  $A$ depending also on the solution itself, as we do here. 

\medskip

We underline we relax the growth restrictions of~\cite{CiMa} allowing to study spaces equipped with modular functions with $L\log L$ or  exponential growth. To obtain existence we need to by-pass tools working in~the reflexive spaces only, employing some ideas of~\cite{pgisazg1}  in the Musielak-Orlicz setting.  The powerful tool we use and find particularly useful is the modular approximation in the classical Orlicz version of Gossez~\cite{Gossez} (see Definition~\ref{def:mod-conv} and Theorem~\ref{theo:approx}) recently adapted to the Musielak-Orlicz case in~\cite{yags}. 

To~establish regularity results we need to apply the embeddings of Orlicz-Sobolev spaces into some Orlicz space, see Section~\ref{sec:emb}.  As a tool we provide   inequalities of~ modular Sobolev-Poincar\'e-type and  Poincar\'e-type, holding with a modular function of arbitrary growth, see Proposition~\ref{prop:Sob-Poi} and Corollary ~\ref{prop:Poincare}, respectively. Once these inequalities are available, we are able to obtain two types of level sets estimates giving regularity properties for the solutions in the scale of Orlicz-Marcinkiewicz spaces.

Since many parts of~our framework (in particular the approximation method) require $\Omega$ to have a regular boundary,  we present all of~the results on a bounded Lipschitz domain.

\section{Statements of main results}
For brevity we skip listing here full notation, presented in detail in Section 3.

\medskip

 {Let  $\Omega\subset\rn$, $N\geq 1$, be a bounded Lipschitz domain and a function $A: \Omega \times \R \times \rn\to\rn$.} We shall consider the following set of assumptions.
\begin{itemize}
\item[(A1)] $A(x,z,\xi)$ is a Carath\'eodory's function, i.e. measurable w.r.t. to $x$ and continuous w.r.t. $z$, as well as w.r.t. $\xi$;
\item[(A2)]  for a.e. $x\in \Omega$ and for all $(z,\xi)\in\r\times\rn$, the following growth conditions hold
\begin{equation}\label{bound}
d_0 B(|\xi|)\leq A(x,z,\xi)\cdot \xi,
\end{equation}
\begin{equation} \label{upper bound}
|A(x,z,\xi)| \leq \frac{1}{3d}\left[ \wt{B}^{-1}(B(|\xi|)) + \wt{P}^{-1}(B(z)) +K(x)\right],
\end{equation}
where   $B:\rp\to\rp$ and  $P:\rp\to\rp$ are two $N$-functions such that 
$P < < B$, $\widetilde{B}$ is the conjugate  of $B$, $\widetilde{B}^{-1}$ is the inverse of $\widetilde{B}$
and $K(x)$ is a function belonging to ${E}_{\wt{B}}(\Omega)$, the closure of $L^{\infty}$ in the $L_{\wt{ B}}$-norm.
\item[(A3)$_w$]  $A(x,z,\xi)$ is  monotone in the last variable, i.e.
\[(A(x,z,\xi) - A(x,z, \eta)) \cdot (\xi-\eta)\ge 0 \]
for a.e  $x\in\Omega$, for every $z\in\r$ and all $\xi,\eta\in\rn$;
\item[(A3)$_s$]  $A(x,z,\xi)$ is strictly monotone in the last variable, i.e.
\[(A(x,z,\xi) - A(x,z, \eta)) \cdot (\xi-\eta)> 0 \]
for a.e  $x\in\Omega$, for every $z\in\r$ and all $\xi\neq\eta\in\rn$;
\item[(A4)]  for a.e $x\in\Omega$ and for $z\in\r$, it holds
\[A(x,z,0)=0 \]
\end{itemize}
Note that conditions (A1)--(A3) are generalizations of the classical Leray-Lions conditions to the Orlicz-Sobolev space setting.

\medskip

We consider the problem
\begin{equation}\label{eq:main:mu}
\left\{\begin{array}{cl}
-\dv A(x,u,\nabla u)= \mu &\qquad \mathrm{ in}\qquad  \Omega,\\
u(x)=0 &\qquad \mathrm{  on}\qquad \partial\Omega,
\end{array}\right.
\end{equation} 
where $\mu\in \mathcal{M}(\Omega)$ is a Radon measure with bounded total variation $|\mu|(\Omega)<\infty$ or $\mu$ is replaced by $f\in L^1(\Omega)$.

\bigskip

To define the solution we need to recall  the truncation  $T_k(u)$  defined as 
\begin{equation}T_k(u)=\left\{\begin{array}{ll}u & |u|\leq k,\\
k\frac{u}{|u|}& |u|\geq k,
\end{array}\right. \label{Tk}
\end{equation}
 {and   the following notation }
\begin{equation}
\label{sp-tr}
{\cal T}^{1,B}(\Omega)=\{u\text{ is measurable in }\Omega :\ T_t(u)\in W^{1,B}(\Omega)\text{ for every }t>0\}.
\end{equation}
The Orlicz-Sobolev space $W^{1,B}$ is defined in Section~\ref{sec:prel}.

\bigskip 

\begin{definition}\label{def:as:mu}  A function $u\in{\cal T}^{1,B}(\Omega)$  is called an \textbf{approximable solution} to the Dirichlet problem~\eqref{eq:main:mu} with given $\mu\in \mathcal{M}(\Omega)$, if there exist a sequence $\{f_k \}_k\subset L^1 (\Omega) $ such that $f_k\xrightharpoonup{*} \mu$ weakly-* in the space of measures, namely that  it holds
\begin{equation}
\label{eq:def:weak-meas-conv}
\lim_{k\to\infty}\int_\Omega\vp\,f_k\,dx=\int_\Omega\vp\,d\mu
\end{equation}
for every $\vp\in C_c(\Omega)$ and a sequence of~weak solutions $\{u_k\}_k\subset W^{1,B}_0 (\Omega)$ to problem~\eqref{eq:main:mu} with $\mu$ replaced by $f_k$, satisfying $u_k\to u$  {a.e. in }$\Omega$. 
\end{definition}

\begin{definition}\label{def:as:f} A  function $u\in {\cal T}^{1,B}(\Omega)$
 is called an \textbf{approximable solution} to the Dirichlet problem~\eqref{eq:main:mu} with $\mu$ replaced by $f\in L^1(\Omega)$, if there exist a sequence  $\{f_k \}_k\subset L^1 (\Omega)\cap (W^{1,B}_0(\Omega))^\prime$   such that $f_k\to f$ in $L^1 (\Omega)$ and a sequence of~weak solutions $\{u_k\}_k\subset W^{1,B}_0 (\Omega)$ to problem~\eqref{eq:main:mu} with $\mu$ replaced by $f_k$, satisfying $u_k\to u$ a.e. in $\Omega$. 
\end{definition}

{ It may happen that an approximable solution is not weakly differentiable. However, it is associated with a vector-valued function on $\Omega$ playing the role of  its gradient on every level of truncation and therefore,  with some abuse of notation, we will still use the symbol $\nabla u$. More details on this issue can be found in Section~\ref{sec:prel}.

\bigskip

Our main results state as follows.

\begin{theorem}\label{theo:main-mu} 
 {Consider a} measure $\mu\in \mathcal{M}(\Omega)$ and a function $A: \Omega \times \R \times \rn\to\rn$ satisfying assumptions (A1), (A2), (A3)$_w$ and (A4). Then there exists an approximable solution $u\in{\cal T}^{1,B}(\Omega)$ to the problem
\eqref{eq:main:mu} and moreover
\begin{equation}\label{A-conv}
A(u,T_t u_k ,\nabla T_{t}(u_k)))\xrightharpoonup[k\to\infty]{*} A(x,T_t u ,\nabla T_tu)\qquad \text{in }\ L_{\widetilde{B}}.
\end{equation}
\end{theorem} 

\begin{theorem}\label{theo:main-f}
Assume  $f\in L^1(\Omega)$ and  $A: \Omega \times \R \times \rn\to\rn$ a function satisfying assumptions (A1), (A2), (A3)$_s$ and (A4). Then there exists a \textbf{unique} approximable solution $u\in{\cal T}^{1,B}(\Omega)$ to the problem~\eqref{eq:main:mu} with $\mu$ replaced by $f$ and~\eqref{A-conv} holds. 
\end{theorem} 

Uniqueness in this context means that the solution does not depend on the choice of approximate problems. Consequently, for the problem with regular data the unique approximable solution agrees with the weak solution, which is trivially also an approximable
solution.

\medskip

  As announced in the Introduction, we shall also obtain some regularity results for  the solution and its gradient. For their statements and  proofs we refer to Section~\ref{sec:reg} since they can be deduced by propositions which are interesting by themselves.

\section{Preliminaries}\label{sec:prel}

\subsection{Notation and basic lemmas }
 Throughout the paper $\Omega$ is a bounded Lipschitz domain in $\rn$, $N\ge1$. We shall use the notation $|\cdot|$ for  the absolute value, as well as for the norm in $\rn$ (for gradient norm)
and denote by $\mathds{1}_{A}$  the characteristic function of a set $A$. 

\medskip

\noindent  Let us start with  two useful results. 
\begin{lemma}[e.g. Lemma~9.1, \cite{GIOV}] \label{lem:ae}
If $g_n:\Omega\to \R$ are measurable functions converging 
to $g$ almost everywhere, then for each regular value $t$ 
of the limit function $g$ we have $\mathds{1}_{\{t<|g_n|\}}\xrightarrow[n\to\infty]{}\mathds{1}_{\{t<|g|\}}$ a.e. in $\Omega$.
\end{lemma}

\begin{lemma} \label{lem:TM1}
 Suppose $w_n\xrightharpoonup[n\to\infty]{}w$ in $L^1(\Omega)$, $v_n,v\in L^\infty(\Omega)$, and $v_n\xrightarrow[n\to\infty]{a.e.}v$. Then \[\int_\Omega w_n v_n\,dx \xrightarrow[n\to\infty]{}\int_\Omega w v\,dx.\]
\end{lemma}

\subsection{The Orlicz setting}

We refer the interested reader to ~\cite{rao-ren} for an exhaustive treatment of the theory of  Orlicz spaces and to~\cite{adams-fournier} for compact, though capturing the point, description of the necessary properties of the Orlicz-Sobolev spaces.

Recall that a function   $B: \rp \to\rp$ is called an $N$-function if  $B$ is a strictly increasing convex function satisfying \[\lim_{r\to 0} \frac{B(r)}{r}=0\qquad\text{and}\qquad \lim_{r\to \infty} \frac{B(r)}{r}=\infty.\]

\noindent Its conjugate function $\widetilde{B}:\rp\to\rp$  is defined by
\begin{equation}\label{def:conj}
\widetilde{B}(s)=\sup_{t>0 }(t\cdot s-B(t))
\end{equation}
and is an $N$-function as well.  

 {Given two $N$-functions $P$ and $B$, we shall write $P<<B$ in order to mean that for each $ \varepsilon>0$ , $P(t)/B(\varepsilon t)\to 0$ as $t\to \infty$. }

 {Observe that one has $P<<B$ if and only if $\widetilde{B}<<\widetilde{P}$, see~\cite{Gossez2}.}

\begin{definition}[The function spaces]\label{def:Osp} Let $B$ be an $N$-function.  We deal with the three  Orlicz classes of functions.  \begin{itemize}
\item[i)]${\cal L}_B(\Omega)$  - the generalised Orlicz class is the set of all measurable functions $\xi$ defined on $\Omega$ such that
\[\int_\Omega B(| \xi(x)|)\,dx<\infty.\]
\item[ii)]${L}_B(\Omega)$  - the generalised Orlicz space is the smallest linear space containing ${\cal L}_B(\Omega )$, equipped with the Luxemburg norm 
\[||\xi||_{L_B}=\inf\left\{\lambda>0:\int_\Omega B\left( \frac{|\xi(x)|}{\lambda}\right)\,dx\leq 1\right\}.\]
\item[iii)] ${E}_B(\Omega)$  - the closure in $L_B$-norm of the set of bounded functions.
\end{itemize}
\end{definition}
Then 
\[{E}_B(\Omega)\subset {\cal L}_B(\Omega)\subset { L}_B(\Omega)\]
and without growth restrictions the inclusions can be proper.

\begin{remark} If $B$ is an $N$-function and $\widetilde{B}$ its conjugate, we have\begin{itemize}
\item the Fenchel-Young inequality \begin{equation}
\label{inq:F-Y}|\xi\cdot\eta|\leq B(|\xi|)+\widetilde{B}(|\eta|)\qquad \mathrm{for\ all\ }\xi,\eta\in\rn.
\end{equation}
\item the generalized H\"older's inequality \begin{equation}
\label{inq:Holder}
\left|\int_{\Omega} \xi\cdot\eta\,dx\right|\leq 2\|\xi\|_{L_B }\|\eta\|_{L_{\widetilde{B}} }\quad \mathrm{for\ all\ }\xi\in L_B(\Omega),\eta\in L_{\widetilde{B}}(\Omega ).
\end{equation}
\end{itemize}
\end{remark}

Moreover, we shall consider the Orlicz-Sobolev space $W^{1,B}(\Omega)$ defined as follows
\begin{equation*} 
W^{1,B}(\Omega)=\big\{u\in W^{1,1}(\Omega): u,\nabla u\in L_B(\Omega)\big\},
\end{equation*}
where $ \nabla$ denotes the distributional gradient. The space  $W^{1,B}(\Omega)$ is 
endowed with the Luxemburg norm
\begin{equation}\label{eq8}
\|u\|_{W^{1,B}(\Omega)}=\inf\bigg\{\lambda>0 :   \int_\Omega B\bigg( \frac{|  u|}{\lambda}\bigg)dx+\int_\Omega B\bigg( \frac{|\nabla u|}{\lambda}\bigg)dx\leq 1\bigg\}.
\end{equation}
The space $W^{1,B}_0(\Omega)$ is defined as a closure of smooth functions, see~\eqref{W0} below.

If $B$ is an $N$-function, then $\big(W^{1,B}(\Omega), \|u\|_{W^{1,B}(\Omega)}\big)$ is a Banach space.

The space ${E}_B(\Omega)$ is separable and due to \cite[Theorem~8.19]{adams-fournier} we have the duality \[({E}_B(\Omega))'=L_{\widetilde{B}}(\Omega).\]

Recall the space ${\cal T}^{1,B}(\Omega)$ defined in~\eqref{sp-tr}. For every $u\in {\cal T}^{1,B}(\Omega)$ there exists a unique measurable function $Z_u:\Omega\to\rn$ such that\begin{equation}
\label{def:Zu}\nabla (T_t(u))=\mathds{1}_{\{|u|<t\}}Z_u\quad\text{a.e. in } \Omega, \text{ for every }{t>0},
\end{equation}
see \cite[Lemma~2.1]{bbggpv}. Since 
\[u\in W^{1,B}(\Omega)\iff u\in {\cal T}^{1,B}(\Omega)\cap L_B(\Omega)\text{ and } |Z_u|\in {\cal T}^{1,B}(\Omega),\]
for such $u$, we have $Z_u=\nabla u$ a.e. in $\Omega$. Thus,  we call $Z_u$ the generalized gradient of $u$ and, abusing the notation, for $u\in {\cal T}^{1,B}(\Omega)$ we write simply $\nabla u$ instead of $Z_u$.

\bigskip

For the spaces $E_B$ and $L_B$ to coincide, and consequently for their reflexiveness, one has to impose $\Delta_2$-condition on $B$ close to infinity (denoted $\Delta_2^\infty$) . Namely, it has to be assumed that there exists a constant $c_{\Delta_2}>0$ such that
\begin{equation}
\label{D2} B(2s)\leq c_{\Delta_2}B(s)\qquad\text{for }\ s>s_0.
\end{equation}

The spaces equipped with the modular functions satisfying $\Delta_2$-condition close to infinity have strong properties. In particular, we have
 \[E_{\widetilde{B}}  \xlongequal[]{\widetilde{B}\in\Delta^\infty_2}L_{\widetilde{B}}\xlongequal[ ]{ }(E_{B})'\xlongequal[]{B \in\Delta^\infty_2} (L_B)'.\]
Moreover, when $B\in\Delta^\infty_2$, then modular and strong convergence coincide. 
 
\medskip

\noindent We would like to stress that we face the problem \textbf{without} this structure. This allows us to deal with a broader class of modular functions. 
Let us discuss the typical assumption of $\Delta_2$-condition, which we do not impose. 

It describes the speed and the regularity of the growth of the function. For example, when we take $B(s) = (1+|s|)\log(1+|s|)-|s|$, its conjugate function is given by $\widetilde{B}(s) = \exp(|s|)-|s|-1$. Then $B\in\Delta_2$ and $\widetilde{B}\not\in\Delta_2$.

 We point out that despite the typical condition \begin{equation}
\label{iBsB}  1<i_B=\inf_{t>0}\frac{tB'(t)}{B(t)}\leq \sup_{t>0}\frac{tB'(t)}{B(t)}=  s_B<\infty
\end{equation} is often treated as equivalent to $B,\wt{B}\in\Delta_2$, as well as to comparison with power-type functions. Nonetheless, the assumption~\eqref{iBsB} is more restrictive, as it requires both regularity of the growth and restricts its speed. Indeed, if $i_B>1$ then $B\in\Delta_2$, whereas $s_B<\infty$ entails the $\Delta_2$-condition imposed on $\wt{B}$. When $B$ satisfies~\eqref{iBsB}, then
\[
\frac{B(t)}{t^{i_B}}\quad\text{is non-decreasing}\qquad\text{and}\qquad\frac{B(t)}{t^{s_B}}\quad\text{is non-increasing}.
\]Moreover, $B(s)=(1+|s|)\log(1+|s|)\in\Delta_2,$ but $i_B=1$. On the other hand, the following example  shows that comparison with two power-type functions is not enough for $\Delta_2$-condition. Another construction can be found in~\cite{BDMS}.

\begin{example}\label{ex:poly-non-D2} For arbitrary $1<p<q<\infty$, there exists a continuous, increasing, and convex function $B:[0,\infty)\to[0,\infty)$ which is trapped between power type functions $t^p$ and $ t^q$ and $B$ {\bf does not} satisfy $\Delta_2$-condition, nor~\eqref{iBsB}. 
\end{example}
\begin{proof} We shall construct $\{a_i\}_{i\in\N}$ and $\{b_i\}_{i\in\N}$ such the desired  function is given by the following formula \[B(t)=\left\{\begin{array}{ll}\text{affine}& t\in (a_i,b_i),\\
t^p& \text{otherwise.} 
\end{array}\right.\] 
To describe $\{a_i\}_{i\in\N}$ let us introduce yet another sequence $\{k_i\}_{i\in\N}$ and fix $a_i=2^{k_i}$ for every $i\in\N$. Let $k_1\in\N$ be large enough to satisfy both
\begin{equation}
\label{k1}
k_1>2^p\qquad\text{and}\qquad \left(\frac{k_1-1}{q}\right)^\frac{1}{k_1}\leq 2^{q-p}.
\end{equation}
Define \[B(t)=2^{pk_1}+2^{(p-1)k_1}(k_1-1)(t-2^{k_1})\qquad\text{for}\qquad t\in (a_1,b_1),\] where $b_1>a_1$ is an intersection point of  chord $f_1(t)=2^{pk_1}+2^{(p-1)k_1}(k_1-1)(t-2^{k_1})$ and $t\mapsto t^p$. Note that~\eqref{k1}$_1$ ensures that
\[2^{pk_1}+2^{(p-1)k_1}(k_1-1)(2^{k_1+1}-2^{k_1})=k_1 2^{pk_1}>(2^{k_1+1})^p,\]
so in particular $2^{k_1+1}<b_1$ and $f(2^{k_1+1})=k_1 2^{pk_1}.$
On the other hand,~\eqref{k1}$_2$ implies that  the slope of the line given by $f_1$ equals $2^{(p-1)k_1}(k_1-1)$ and is smaller than the derivative of $ t^q$ in $a_1$. Combining it with $t^p|_{2^{k_1}}<t^q|_{2^{k_1}}$ we get that $B(t)<t^q$ on $(a_1,b_1)$.

Let $k_2$ be the smallest natural number such that $a_2=2^{k_2}\geq b_1$ and set $B(t)=t^p$ on $(b_1,a_2)$. We repeat the construction of chord. Note that since $k_2>k_1$, the condition~\eqref{k1} with $k_1$ substituted with $k_2$ is satisfied. Thus, the chord is trapped between $t^p$ and $ t^q$. Iterating further the construction we obviously obtain a continuous, increasing, and convex function, whose graph lies between the same power-type functions. Moreover, we also get the sequences  $\{a_i\}_{i}$, $\{b_i\}_{i}$, and $\{k_i\}_{i}$ such that $k_i\to\infty$,  $2a_i<b_i\leq a_{i+1}$ and
\[B(a_i)=a_i^p\qquad\text{and}\qquad B(2a_i)=k_i a_i^p=k_i B(a_i),\]
which contradicts with $\Delta_2$-condition.  Moreover, one can check that $i_B=1$, which violates~\eqref{iBsB}. \qed
\end{proof}

\subsection{The topologies } \label{ssec:top}

We shall distinguish topology $\sigma( L_B, L_{\widetilde{B}})$ from weak-* topology in $L_B$, namely $\sigma( L_B, E_{\widetilde{B}})$.\\

\smallskip

\noindent We say that $\{u_n\}\subset L_B$ is $\sigma( L_B, L_{\widetilde{B}})$-convergent to $\in L_B$, if for any $v\in L_{\wt{B}}$ \[\int u_n\, v\,dx\xrightarrow[n\to\infty]{}\int u\, v\,dx.\]
We say that $\{u_n\}\subset L_B$ is weakly-*  convergent to $u\in L_B$, if for any $v\in E_{\wt{B}}$ \[\int u_n\, v\,dx\xrightarrow[n\to\infty]{}\int u\, v\,dx.\]
We say that $\{u_n\}_n$ converges to $u$ in norm (strongly) in $L_B(\Omega)$, if
$\|u_n-u\|_{L_B(\Omega)}\rightarrow 0$ as $n\rightarrow \infty$. 

\medskip

Obviously strong convergence implies both weak-type convergences above, but there is one more intermediate type of convergence more relevant in this setting.
 
\begin{definition}[Modular convergence]\label{def:mod-conv} A sequence $\{u_\delta\}_\delta$ is said to converge modularly to $u$ in $L_B(\Omega)$\\ if there exists a parameter $\lambda>0$ such that $ \int_{\Omega}B\left({|u_\delta-u|}/{\lambda}\right)\, dx\rightarrow 0$ as $\delta\rightarrow 0$,
 equivalently\\ if there exists a parameter $\lambda>0$ such that $\left\{B\left( {|u_\delta|}/{\lambda}\right)\right\}_\delta \ \text{is uniformly integrable in } L^1(\Omega)$ {and} $ u_\delta\xrightarrow[\delta\to 0]{} u$ {in measure}.
\end{definition}

Following Gossez \cite{Gossez}, we define the space \begin{equation}
 \label{W0}
 W^{1,B}_0(\Omega)=\overline{\mathcal{C}^{\infty}_{0}(\Omega)}^{\sigma( L_B,  E_{\widetilde{B}})}
 \end{equation} i.e. as the 
closure of $\mathcal{C}^{\infty}_{0}(\Omega)$ in $W^{1,B}(\Omega)$ with respect to the topology $\sigma( L_B,  E_{\widetilde{B}})$.  
Naturally ${\cal T}_0^{1,B}(\Omega)$ is defined as ${\cal T}^{1,B}(\Omega)$ in~\eqref{sp-tr}  replacing $W^{1,B}(\Omega)$ with $W_0^{1,B}(\Omega)$.
\bigskip

We write 
\begin{equation*}
{u_\delta \mconvd u\ \text{ in }\ W^{1,B}(\Omega)\qquad\Longleftrightarrow\qquad \left(u_\delta \mconvd u\quad\text{and}\quad  \nabla u_\delta\mconvd\nabla u\ \text{ in } L_B(\Omega)\right)}.
\end{equation*}

{We will use the following approximation in the modular topology due to Gossez}. Note that the final  boundedness of the norm results from the original proof.

\begin{theorem}[cf. \cite{Gossez}, Theorem~4]\label{theo:approx} Suppose $\Omega\subset\rn$, $N\geq 1$ is a bounded Lipschitz domain and $u\in W_0^{1,B}(\Omega)$. Then there exists a sequence $\{u_\delta\}_{\delta}\in C_0^\infty(\Omega)$ such that $u_\delta\xrightarrow[\delta\to\infty]{mod} u$ in $W^{1,B}(\Omega).$ 

 Moreover, if $u\in{L^\infty(\Omega)}$, then  $\|u_\delta\|_{L^\infty(\Omega)}\leq c(\Omega) \|u\|_{L^\infty(\Omega)}$.
\end{theorem}

Because of the notion of the modular convergence the fundamental role in the theory is played by the following classical results.
\begin{theorem}[Vitali Convergence Theorem]\label{theo:VitConv} Let $(X,\mu)$ be a positive measure space, $\mu(X)<\infty $, and $1\leq p<\infty$. If $\{u_{n}\}$ is uniformly integrable in $L^p_\mu$,   $u_{n}\to u$ in measure  and $|u(x)|<\infty $  a.e. in $X$, then  $u\in  {L}^p_\mu(X)$
and  $u_{n}\to u$ in  ${L}^p_\mu(X)$.
\end{theorem} 
\begin{theorem}[de la Vallet Poussin Theorem]\label{theo:delaVP}
Let $B$ be an $N$-function and $\{u_n\}$ be a sequence of measurable functions such that  $\sup_{n\in\N}\int_\Omega B(|u_n(x)|)dx<\infty$. Then the sequence $\{u_n\}_n$ is uniformly integrable.
\end{theorem}

In general, if $u_\delta\xrightarrow[\delta\to 0]{}  u$ in norm in $L_B$, then $u_\delta\xrightarrow[\delta\to 0]{mod} u$ and not conversely. Nonetheless, the reverse implication can be obtained via the following lemma.

\begin{lemma}\label{lem:modular-norm}
Let $B$ be an $N$-function and  $u_n\xrightarrow[n\to\infty]{mod} u$ in $L_B(\Omega)$ with every $\lambda>0$, then $u_n\xrightarrow[n\to\infty]{} u$ in the norm topology in $L_B(\Omega)$.
\end{lemma}
\begin{proof}
We present the proof for $u\equiv 0$ only.

 If $\int_\Omega B(\lambda u_n)dx\xrightarrow[n\to\infty]{} 0$, then for every $\lambda>0$ there exists $n_\lambda$, such that for every $n>n_\lambda$ we have $\int_\Omega B(\lambda u_n)dx\leq 1/\lambda.$ Therefore, for every $n>n_\lambda$ also $\|u_n\|_{L_B(\Omega)}\xrightarrow[n\to\infty]{} 0$. On the other hand, if $\|u_n\|_{L_B(\Omega)}\xrightarrow[n\to\infty]{} 0$, then for any fixed $\lambda>0$ we get $\|\lambda u_n\|_{L_B(\Omega)}\xrightarrow[n\to\infty]{} 0$. This means that for every $\ve \in (0,1)$ there exists $n_\ve,$ such that for every $n>n_\ve$ it holds that $\|\lambda u_n\|_{L_B(\Omega)}<\ve<1$.

 Since for arbitrary $\xi\in L_B(\Omega)$ with $\|\xi\|_{L_B}\leq 1$, we have $\int_\Omega B(\xi(x))\,dx\leq \|\xi\|_{L_B(\Omega)}$. Therefore, $\int_\Omega B(\lambda u_n)dx\leq \|\lambda u_n\|_{L_B(\Omega)}<\ve$ for every $n>n_\ve$, which implies the claim.\qed
\end{proof}

\begin{lemma}[Lemma~6, \cite{Gossez}]\label{modw}
Let $u_n$, $u\in L_{B}(\Omega)$. If $u_n\xrightarrow[n\to\infty]{} u$ modularly, then $u_n\to u$ in $\sigma(L_B,L_{\tilde B})$.
\end{lemma}
Note nonetheless, that for $B\in\Delta_2$, the weak and modular closures are equal.

\begin{lemma}[Weak-strong convergence]\label{lem:weak-strong-conv}
Assume that $\{u_n\}_n\subset E_{\wt{B}}$ and $\{v_n\}_k\subset L_{B}$  are sequences such that
 $$u_n\to u\in E_{\wt{B}}\qquad \qquad v_n\xrightharpoonup* v\in L_{B}.$$  Then
$$\int_{\Omega}u_nv_n\, dx\to \int_{\Omega}uv\, dx.$$
\end{lemma}
\medskip
\begin{proof}
We write
\begin{eqnarray*}
\int_{\Omega}(u_nv_n-uv)\, dx=\int_{\Omega}(u_n-u)v_n\, dx+\int_{\Omega}u(v_n-v)\, dx.
\end{eqnarray*}
Then, by H\"older's inequality \eqref{inq:Holder}  we have
$$\left|\int_{\Omega}(u_nv_n-uv)\, dx\right|\le ||u_n-u||_{L_{\wt{B}}}||v_n||_{L_B}+\left|\int_{\Omega}u(v_n-v)\, dx\right|$$
 and therefore, the result follows observing that $||v_n||_{L_B}$ is uniformly bounded due to the assumption $v_n\xrightharpoonup* v\in L_{B}$ and that $(E_{\wt{B}})'=L_B$.\qed
\end{proof}

\section{Sobolev-type Embeddings}\label{sec:emb}

\medskip

To establish regularity result we need to apply the results on embedding of the Orlicz-Sobolev spaces into some Orlicz space, namely
\[W_0^{1,B}(\Omega)\hookrightarrow{} L_{\hat{B}}(\Omega),\]
with $\hat{B}$ growing in a certain sense faster than $B$.  We use two types of results, which -- to be distinguished -- will be roughly called the optimal and the easy one. The optimal embedding proven by Cianchi~\cite{Ci97} distinguishes two cases: of quickly and slowly growing modular function $B$, corresponding to the cases of $p$-Laplacian with $p>n$ and $p\leq n$.  The easy embedding, which yields that $W^{1,B}_0 (\Omega)\hookrightarrow{} L_{B^{N'}}(\Omega)$ is provided below the optimal one. It is weaker than the optimal, but it is easy and captures a general $N$-function $B$ independently of any growth conditions. Let us stress that since the rest of our framework requires $\Omega$ to be a Lipschitz bounded domain, we present all of the results on such domains. See e.g.~\cite{CiMa} for an overview on the issue of the regularity of the boundary in relation to the embedding.

To apply the optimal embeddings we employ, we note that in~\cite{Ci97}  the Sobolev inequality is proven under the restriction 
\begin{equation}\label{int0B}\int_0\left(\frac{t}{B(t)}\right)^\frac{1}{N-1}dt<\infty, 
\end{equation} 
concerning the growth of $B$ in the origin. Nonetheless, the properties of $L_B$ depend on the behaviour of $B(s)$ for large values of $s$ and~\eqref{int0B} can be easily by-passed in application. Indeed, if it would be necessary for~\eqref{int0B} we shall substitute $B(t)$  by ${B}^0(t)=tB(1)\mathds{1}_{[0,1]}(t)+B(t)\mathds{1}_{(1,\infty)}(t)$. 

The conditions
\begin{equation}
\label{intB}
\int^\infty\left(\frac{t}{B(t)}\right)^\frac{1}{N-1}dt=\infty \qquad\text{and}\qquad
\int^\infty\left(\frac{t}{B(t)}\right)^\frac{1}{N-1}dt<\infty,
\end{equation}
roughly speaking, describe slow and fast growth of $B$ at infinity respectively. 

For $N'= {N}/({N-1})$, we consider 
\begin{equation}\label{BN}
H_N(s)=\left(\int_0^s\left(\frac{t}{B(t)}\right)^\frac{1}{N-1}dt\right)^\frac{1}{N'},\qquad
B_N(t)=B(H_N^{-1}(t)),\qquad\text{and}\qquad
\phi_N(s)= (H_N(s))^{N'}.
\end{equation}

When the integrability in the origin condition~\eqref{int0B} is satisfied and the growth of $B$ at infinity is slow, that is when~\eqref{intB}$_1 $ holds, then~\cite[Theorem~3]{Ci97} provides the following continuous embedding
\begin{equation}
\label{W1LBinLBN}
W_0^{1,B}(\Omega)\hookrightarrow{} L_{B_N}(\Omega),
\end{equation}
where $B_N$ is given by~\eqref{BN}. Otherwise, when the growth of $B$ a infinity is fast, that is when~\eqref{intB}$_2$ holds, 
then we have the following continuous embedding 
\begin{equation}
\label{W1LBinLinf}
W^{1,B}(\Omega)\hookrightarrow{} L^\infty(\Omega).
\end{equation} 
This result was proven first in~\cite{Ta89}, see also~\cite{Ci96}.

\bigskip
 
In the general case, independently of the growth conditions we provide the easy embedding \[W^{1,B}_0 (\Omega)\hookrightarrow{} L_{B^{N'}}(\Omega).\] 
More precisely, we prove the following

 \begin{proposition}[The Sobolev-Poincar\'{e} inequality without growth restrictions]\label{prop:Sob-Poi}
Let $\Omega$ be a bounded Lipschitz domain in $\rn$, $N\geq 1$ and $B$ be an $N$-function. There exist constants $c_1,c_2>0$ depending on $\Omega$, such that 
for every $u\in W_0^{1,B}(\Omega)$ 
\[\left(\int_\Omega B^{N'}(c_1|u|)dx\right)^\frac{1}{N'}\leq c_2 \int_\Omega B ( |\nabla u|)dx.\]
 \end{proposition}
Before giving the proof of the above Proposition,  let us observe that  as a direct consequence,  by the use of the H\"older inequality, we can easily obtain the following Poincar\'e-type inequality
 \begin{corollary}[The modular Poincar\'{e} inequality]\label{prop:Poincare}
Let $\Omega$ be a bounded Lipschitz domain in~$\rn$, $N\geq 1$ and $B$ be an $N$-function. There exist constants $c_1,c_2>0$ depending on $\Omega$,  such that 
for every $u\in W_0^{1,B}(\Omega)$ 
\[\int_\Omega B (c_1|u|)dx\leq c_2 \int_\Omega B ( |\nabla u|)dx. \]
 \end{corollary}

In the proof of Proposition~\ref{prop:Sob-Poi} we will use the following version of the H\"older inequality.
\begin{lemma}\label{lem:Hold-ext} Suppose $Q^N=[-1,1]^N$ and $f_i\in L^{N-1}(Q^{N-1})$, then
\[\int_{Q^N}\prod_{i=1}^N|f_i|dx\leq \prod_{i=1}^N\left(\int_{Q^{N-1}} |f_i|^{N-1}dx'\right)^\frac{1}{N-1}.\]
\end{lemma}

\noindent{Proof of Proposition~\ref{prop:Sob-Poi}. } The proof consists of three steps starting with the case of smooth and compactly supported functions on small cube, then turning to the Orlicz class and concluding the~claim on arbitrary bounded set.

\medskip

{\bf Step 1.} We start the proof for $u\in C_0^\infty(\Omega)$ with $\supp\, u\subset\subset [- {1}/{4}, {1}/{4}]^N$. Let $u$ be extended by $0$ outside~$\Omega$. Note that for every $j=1,\dots,N$
\[ |u(x)|\leq \int_{-\frac{1}{2}}^{\frac{1}{2}}  |\partial_j u(x)|dx_j .\] 
Applying $B^{1/(N-1)}$, which is increasing, to both sides above and using Jensen's inequality,  we get
\[B^\frac{1}{N-1}(|u(x)|)\leq B^\frac{1}{N-1}\left(\int_{-\frac{1}{2}}^{\frac{1}{2}} |\partial_j u(x)|dx_j\right)\leq \int_{-\frac{1}{2}}^{\frac{1}{2}}   B^\frac{1}{N-1}\left( |\partial_j u(x)|\right)dx_j.\]
When we multiply $N$ copies of the above inequality, integrate over $\Omega$, and apply Lemma~\ref{lem:Hold-ext}, we obtain
\[\begin{split}
\int_\Omega B^\frac{N}{N-1}(|u(x)|)dx &=\int_{Q^N} B^\frac{N}{N-1}(|u(x)|)dx \leq \int_{Q^N} \prod_{i=1}^N\int_{-\frac{1}{2}}^{\frac{1}{2}}  B^\frac{1}{N-1}\left(|\partial_j u(x)|\right)dx_j\,dx\\
&\leq  \prod_{i=1}^N  \left(  \int_{Q^{N-1}} \int_{-\frac{1}{2}}^{\frac{1}{2}} B\left(|\partial_j u(x)|\right) dx_j\, dx'\right)^\frac{1}{N-1}= \prod_{i=1}^N  \left( \int_{-\frac{1}{2}}^{\frac{1}{2}} \int_{Q^{N-1}}  B\left(|\partial_j u(x)|\right) dx'\, dx_j\right)^\frac{1}{N-1}\leq\\
&\leq   \left(\int_{Q^N} B\left(|\nabla u(x)|\right) dx\right)^\frac{N}{N-1}= \left(\int_{\Omega} B\left(|\nabla u(x)|\right) dx\right)^\frac{N}{N-1}.\end{split}\]
 
\medskip

{\bf Step 2.}  Let $u\in W_0^{1,B}(\Omega)$. Then by Theorem~\ref{theo:approx} there exists a sequence $\{u_\delta\}_\delta\subset C_0^\infty(\Omega)$ such that \[u_\delta \mconvd u\ \text{ in }\ W^{1,B}(\Omega).\]
Note that $\{u_\delta\}_\delta$ is a Cauchy sequence in the modular topology in $W^{1,B}(\Omega)$ and the  inequality obtained above holds for every $u_\delta$. Moreover, $\{u_\delta\}_\delta$ is also a Cauchy sequence in the modular topology in $L_{B^{N'}}(\Omega)$.

 Due to the modular convergence we get $\nabla u_\delta \to \nabla u$ in measure. Jensen's inequality and properties of modular convergence together with the Lebesgue Dominated Convergence Theorem enable to pass to the limit with $\delta\to 0$ to get the final claim on the small set $\Omega$. 

\medskip

{\bf Step 3.} Suppose that $\Omega$ is arbitrary bounded set containing $0$. It is contained in the cube of~the~edge $D={\rm diam} \Omega$. Then $\wt{u}(x)=u\left(4dx\right)$ has ${\rm supp}\,\wt u\subset \Omega_1\subset \left[-\frac{1}{4},\frac{1}{4}\right]^N.$ We have
\[\left(\int_\Omega B^{N'}(|u|)dx\right)^\frac{1}{N'}=\left((4D)^N\int_{\Omega_1} B^{N'}(|\wt{u}|)\right)^\frac{1}{N'}dx\leq (4D)^\frac{N}{N'} \int_{\Omega_1} B(|\nabla\wt{u}|)dx=\frac{1}{4D} \int_{\Omega} B(4D|\nabla {u}|)dx.\]
 To obtain the estimate on an arbitrary domain we need only to observe that the Lebesgue measure is translation-invariant.
\qed

\section{Main proofs}\label{sec:main-proof}
 {This section is devoted to the proofs of our main results which will be splitted into  different steps.   We start with showing the existence of solutions $u_k$ to problems with regular and bounded data by using the general known theory. In the second and in the third steps we respectively obtain uniform a priori estimates for such weak solutions and almost every where convergence of $u_k$ to some $u$. Step~4 provides that this limit $u$ is the desired approximable solution. Finally  in Step~5 we pass to measure data.}

\medskip

{The first subsection is dedicated to the monotonicity trick which will be instrumental for our arguments.}

\subsection{{Monotonicity trick}} 
Note that the idea of this trick was used in~\cite{Gossez2,MT} in a very general situation. We present it together with the proof for the sake of completeness. 
\begin{proposition}[Monotonicity trick]\label{prop:monotonicity-trick}
Suppose $A$ satisfies conditions  (A1) and (A2).

Assume further that there exists ${\cal A}\in L_{\wt{B}}(\Omega)$ such that for some $v\in W^{1,B}_0(\Omega)\cap L^\infty(\Omega)$ it holds 
\begin{equation}
\label{anty-mon}
\int_\Omega \big(A(x,v,\zeta)-{\cal A}\big)\cdot(\zeta -\nabla v)\,dx\geq0 \quad  \forall \zeta\in L_B(\Omega).
\end{equation}
Then
\[A(x,v,\nabla v)={\cal A}\qquad\text{a.e. in }\ \Omega.\]
\end{proposition}

\begin{proof} Let us define\begin{equation*}
\Omega_m=\{x\in\Omega:\ |\nabla v|\leq m\}.
\end{equation*} 
 Fix arbitrary $0<j<i$ and notice that $\Omega_j\subset\Omega_i$.

We consider~\eqref{anty-mon} with
\[\zeta=(\nabla v)\mathds{1}_{\Omega_i}+h\vec{w}\mathds{1}_{\Omega_j},\]
where $h>0$ and $\vec{w}\in L^\infty(\Omega;\rn)$, namely
\begin{equation*}
\int_\Omega\Big( A\big(x,v,(\nabla v)\mathds{1}_{\Omega_i}+h\vec{w}\mathds{1}_{\Omega_j}\big)-{\cal A} \Big)\cdot\Big( (\nabla v)\mathds{1}_{\Omega_i}+h\vec{w}\mathds{1}_{\Omega_j}-\nabla v\Big)\, dx\geq  0.
\end{equation*} 
Notice that it is equivalent to
\begin{equation}
\label{A-po-mon}\begin{split}& -\int_{\Omega\setminus\Omega_i}   (A(x,v,0) - {\cal A})\cdot\nabla v\,dx  +h\int_{ \Omega_j} (A(x,v,\nabla v+h\vec{w})-{\cal A} )\cdot \vec{w}\,  dx\geq  0.\end{split}
\end{equation} 
The first {integral} above {tends } to zero when $i\to\infty$. {Indeed $A(x,v,0)=0$, ${\cal A}\in L_{\tilde B}, \nabla v\in L_{ B}$ and therefore H\"older's inequality gives the boundedness of the integrands in $L^1$. The convergence to zero follows taking}     into account {the} shrinking {domains} of integration.

{It follows that}
\begin{equation*} h\int_{ \Omega_j} \big(A(x,v,\nabla v +h\vec{w})-{\cal A} \big)\cdot \vec{w}\,  dx\geq  0
\end{equation*}
and obviously that
\begin{equation}\label{greater}
 \int_{ \Omega_j} \big(A(x,v,\nabla v +h\vec{w})-{\cal A} \big)\cdot \vec{w}\,  dx\geq  0.
\end{equation}
Note that $\nabla v +h\vec{w}\to \nabla v $ in $L^\infty(\Omega_j)$ as $h\to 0$ and thus
\[A(x,v,\nabla v +h\vec{w})\xrightarrow[h\to 0]{}A(x,v,\nabla v )\quad\text{a.e. in}\quad \Omega_j.\] 
Moreover,  $A(x,v,\nabla v +h\vec{w})$ is bounded on $\Omega_j$. Let $\bar{c}= h ||\vec{w}||_{\infty}$ . Using~(A2) and Jensen's inequality we have that in $\Omega_j$
\[\begin{split} \wt{B}\left(d|A(x,v,\nabla v +h\vec{w})|\right)&\leq   \wt{B}\left(\frac{1}{3} \big|K(x)+ \wt{P}^{-1}(B(|v|))+ \wt{B}^{-1}(B(\bar{c}))\big|\right)\\ 
&\leq   \frac{1}{3}\wt{B}\left( K(x)\right)+\frac{1}{3}\wt{B}\left( \wt{P}^{-1}(B(\|v\|_{L^\infty(\Omega_j)}))\right)+\frac{1}{3}\wt{B}\left( \wt{B}^{-1}(B(\bar{c}))\right)\in L^1(\Omega_j).
\end{split}\]
 Hence, we have uniform boundedness of $\left(\wt{B}\big(A(x,v,\nabla v +h\vec{w})\big)\right)_h$ in $L^1(\Omega)$ and {by Theorem \ref{theo:delaVP} we deduce the uniform}   integrability of $\left(A(x,v,\nabla v +h\vec{w})\right)_h$. {Since} $|\Omega_j|<\infty$ and (A1) implies continuity with respect to the last variable,  we can apply {Theorem \ref{theo:VitConv}} to get
\[A(x,v,\nabla v +h\vec{w})\xrightarrow[h\to 0]{}A(x,v,\nabla v)\quad\text{in}\quad L^1(\Omega_j;\rn).\] 
Thus
\begin{equation*} \int_{ \Omega_j} (A(x,v,\nabla v +h\vec{w})-{\cal A} )\cdot \vec{w} \, dx\xrightarrow[h\to 0]{} \int_{ \Omega_j} (A(x,v,\nabla v)-{\cal A} )\cdot \vec{w} \, dx.
\end{equation*}
{Taking into account \eqref{greater}, it follows that }
\begin{equation*}  \int_{ \Omega_j} (A(x,v,\nabla v)-{\cal A} )\cdot \vec{w}\,  dx\geq 0,
\end{equation*}
for any $\vec{w}\in L^\infty(\Omega;\rn)$.

{If we consider}
\[\vec{w}=\left\{\begin{array}{ll}-\frac{A(x,v,\nabla v)-{\cal A} }{|A(x,v,\nabla v)-{\cal A} |}&\ \text{if}\quad A(x,v,\nabla v)-{\cal A} \neq 0,\\
0&\ \text{if}\quad A(x,v,\nabla v)={\cal A} ,
\end{array}\right.\]
we obtain 
\begin{equation*}  \int_{ \Omega_j} |A(x,v,\nabla v)-{\cal A} |\, dx\leq 0,
\end{equation*}
{and} hence \[A(x,v,\nabla v)={\cal A} \quad \text{a.e.}\quad\text{in}\quad \Omega_j.\]
Since $j$ is arbitrary, we have the equality a.e. in $\Omega$ and~\eqref{anty-mon} is satisfied.\end{proof}\qed

\subsection{Proof of Theorem~\ref{theo:main-mu}}

\noindent {\bf Step 1. Existence of $u_k$ solving approximate problem}

\medskip

\noindent Let us consider $\{f_k\}_k\subset C_0^\infty(\Omega)$, such that
\begin{equation}\label{lim-fk-to-f}
f_k\to f\qquad\text{in}\quad L^1(\Omega)\qquad\text{and}\qquad f_k(x)\leq 2  f(x)\quad \text{a.e. in }\ \Omega.
\end{equation}
We are going to show the existence of a weak solution $u_k$ to the problem 
\begin{equation}\label{prob:trunc}
\left\{\begin{array}{cl}
-\dv A(x,u_k,\nabla u_k)= f_k &\qquad \mathrm{ in}\qquad  \Omega,\\
u_k(x)=0 &\qquad \mathrm{  on}\qquad \partial\Omega,
\end{array}\right.
\end{equation}
Recall that $u_k\in W^{1,B}_0$ is a weak solution to the problem \eqref{intro:ell:f} if 
$$\int_{\Omega}A(x,u_k, \nabla u_k) \nabla \varphi\, dx=\int_{\Omega}f_k\varphi\,dx$$
for every  $\varphi\in W^{1,B}_0$.
Since (A1)-(A3) hold, by Theorem 4.3 in \cite{MT} we have that  the operator is pseudomonotone and therefore, by Theorem 5.1 in \cite{MT}, we get the existence of a distributional solution.

Then, due to the modular approximation (see Theorem~\ref{theo:approx}),  we obtain the existence of $u_k\in W_0^{1,B}(\Omega)$, such that\begin{equation}
\label{weak:u-k}
\int_\Omega A(x,u_k, \nabla u_k)\cdot \nabla \vp\,dx=\lim_{\delta\to 0}\int_\Omega A(x,u_k, \nabla u_k)\cdot \nabla \vp_\delta\,dx=\lim_{\delta\to 0}\int_\Omega f_k\,\vp_\delta\,dx=\int_\Omega f_k\,\vp\,dx
\end{equation}
for every $\vp\in W_0^{1,B}(\Omega)$. 

\medskip

\noindent{\bf Step 2. A priori estimates}

\medskip

 \noindent In order to obtain uniform integrability of sequences $\{A(x,{T_t u_k,}\nabla T_t u_k)\}_{k}$ and $\{\nabla T_t u_k \}_{k}$ we will prove the two following a priori estimates. For $u_k$ being a weak solution to~\eqref{prob:trunc}  and $f\in L^1(\Omega)$, we will have  for any $t>0$
\begin{eqnarray}
\int_\Omega B(|\nabla T_t u_k|)\,dx&\leq& c_0 t \|f\|_{L^1(\Omega)},\label{Bapriori}\\
\int_\Omega \widetilde{B}\left(\frac{1}{d} |A(x,T_t u_k,\nabla T_t u_k)|\right)\,dx&\leq& c_0\,t \|f\|_{L^1(\Omega)}+c_1\,B_s(t)+c_2,\label{B*apriori}
\end{eqnarray}
 where $B_s<<B$ and the constants $c_0, c_1, c_2$  depend only on the growth condition~(A2). More precisely, $c_0=2/d_0$, $c_1=c(B,P,\Omega)$, $c_2=c(K)$.

Due to~\eqref{weak:u-k},  we get
\begin{equation}
\label{1sttest}
\int_\Omega A(x, T_t u_k, \nabla T_t u_k )\nabla T_t u_k\,dx= 
\int_\Omega A(x, u_k, \nabla u_k) \nabla T_t u_k dx 
= \int_\Omega f_k  T_t u_k\,dx\leq 2 t \|f\|_{L^1(\Omega)}.
\end{equation}
{Observe that we used that $A(x,u_k, \nabla u_k) \in L_{\wt{B}}$ and estimate at \eqref{lim-fk-to-f}.}
{Estimate  \eqref{Bapriori} immediately follows by using \eqref{bound} 
\begin{eqnarray}\label{first}
d_0\int_\Omega B( |\nabla T_t u_k|))\,dx\leq \int_\Omega A(x, T_t u_k, \nabla T_t u_k)\nabla T_t u_k\,dx\leq 2 t \|f\|_{L^1(\Omega)}.
\end{eqnarray}
On the other hand, if we use   \eqref{upper bound}, Jensen's inequality and \eqref{Bapriori}, we have
\begin{eqnarray*}
\int_\Omega \widetilde{B}(d|A(x,T_t u_k,\nabla T_t u_k|))\,dx&\leq&
\int_\Omega \widetilde{B}\left(\frac{1}{3}\left[\widetilde{B}^{-1}(B(|\nabla T_t u_k|))+ \widetilde{P}^{-1}(B(|T_t u_k|))+K(x)\right]\right)\,dx\cr\cr
 &\leq&
\frac{1}{3}\int_\Omega \widetilde{B}\left(\widetilde{B}^{-1}(B(|\nabla T_t u_k|))\right)+ \widetilde{B}\left(\widetilde{P}^{-1}(B(|T_t u_k|))\right)+\widetilde{B}\left(K(x)\right)\,dx\cr\cr 
&\leq&\frac{1}{3}
\int_\Omega B(|\nabla T_t u_k|)+c(B,P)B(t)+\widetilde{B}\left(K(x)\right)\,dx\cr\cr &\leq&
c_0 \,t ||f||_{L^1(\Omega)}+c(B,P,\Omega)B_s(t)+c(K).
\end{eqnarray*}
Note that we have used that the assumption  $P<<B$ is equivalent to $\widetilde{B}<< \widetilde{P}$, estimate \eqref{first} and that $K\in E_{\widetilde B}$.}

\bigskip

\noindent{\bf Step 3. Convergence $u_k\xrightarrow{a.e} u$}

\medskip

\noindent The  a priori estimates \eqref{Bapriori}, {the Banach-Alaoglu theorem combined with Dunford-Pettis theorem,} and the fact that $B$ is an $N$-function  imply that for each $t{>0}$ the sequence {$\{T_t u_k\}_k$} is bounded in $W^{1,1}_0(\Omega)$. Moreover, the Poincar\'e inequality from Corollary~\ref{prop:Poincare} and estimate \eqref{Bapriori} ensure that $\{T_t u_k\}_k$ is bounded in $W^{1,B}_0(\Omega)$. Hence, the embedding imply that there exists a function $u$ such that
 
\begin{equation}\begin{split}\label{conv:basic}
T_t u_k&\xrightarrow[k\to\infty]{}  T_t u\quad \text{strongly in } L^1(\Omega),\\
T_t u_k&\xrightarrow[k\to\infty]{}  T_t u\quad \text{a.e.}\\
\nabla T_t(u_k)&\xrightharpoonup[k\to\infty]{ }  \nabla T_t u\quad \text{weakly in } L^1(\Omega),\\
\nabla T_t u_k&\xrightharpoonup[k\to\infty]{*}  \nabla T_t u\quad \text{weakly-$*$ in } L^B(\Omega).\end{split}
\end{equation}
 Since truncated functions converge a.e., for every $t$ fixed and for every $\epsilon$ there exists $\tau$ such that for $k,m$ sufficiently large
\begin{equation}\label{doptro}
| \{ |T_t u_k - T_t u_m| > \tau \} | \leq \epsilon.
\end{equation}
Now observe that for given $t, \tau >0$ we have
$$
| \{ |u_k - u_m| > \tau \} | \leq | \{ |u_k | > t \} | +  | \{ |u_m | > t \} | +  | \{ |T_t u_k - T_t u_m| > \tau \} |
$$
for $k,m \in \mathbb{N}$.

\noindent On the other hand, since $B$ is increasing we get for every $l>0$ 
$$
|\{|u_k|\geq l\}|=|\{|T_l(u_k)|= l\}|=|\{|T_l(u_k)|\geq l\}|=|\{B(c_1|T_l(u_k)|)\geq B(c_1l)\}|,
$$
therefore
\begin{equation}\label{uk>l}
\begin{split}
|\{|u_k|\geq l\}|&\leq \int_{\Omega} \frac{B(|c_1T_l(u_k)|)}{B(c_1l)} dx\leq \frac{c(N,\Omega)}{B(l)}  \int_\Omega B ( |\nabla T_l(u_k)|)dx \\
&\leq 
 \frac{C( N,\Omega)}{B(l)}  \cdot l\,\|f\|_{L^1(\Omega)}\\
 &\leq 
 C(f,B,N,\Omega)\left(\frac{l}{B(l)}\right)\xrightarrow[l\to\infty]{}0. 
\end{split}
\end{equation}
In the above estimates we apply (respectively) the Chebyshev inequality,  Corollary~\ref{prop:Poincare}, the a priori estimate~\eqref{Bapriori}. The limit results from the superlinear growth in the infinity of $N$-function $B$.
 
Therefore, using \eqref{uk>l}, for every $\epsilon$ we can choose $t$ so large that 
$$
 | \{ |u_k | > t \} | < \epsilon \quad \text{and} \quad  | \{ |u_m | > t \} | < \epsilon
$$
and then, recalling also \eqref{doptro}, we obtain  that $u_k$ is a Cauchy sequence in measure. It follows that , up to a subsequence, 
\begin{equation}
\label{conv:ae:uk-to-u}
 u_k \xrightarrow[k\to\infty]{}  u \quad \text{a.e. in }  \Omega ,
\end{equation}
that is $u$ is an approximable solution to our problem.
\medskip

\noindent{{\bf Step 4. Convergence $A(u,T_t u_k ,\nabla T_{t}(u_k)))\xrightharpoonup{*} A(x,T_t u ,\nabla T_tu)$ in $L_{\widetilde{B}}$}}

\medskip

\noindent Since by~\eqref{B*apriori} we have that there exists  $\cA_t\in L_{\widetilde{B}}(\Omega )$ such that
\begin{equation}
\label{a-conv-ca}
A(x,T_t u_k ,\nabla T_{t}(u_k))\xrightharpoonup{*} \cA_t \quad  \text{weakly}-*\ \text{in}\ L_{\widetilde{B}}(\Omega),
\end{equation}
our {first} aim is to prove that
\begin{equation}
\label{lim<A'} \limsup_{k\to\infty} 
\int_\Omega A(x,T_t u_k, \nabla T_{t}u_k) \nabla T_t u_k\, dx=\int_\Omega {\cal A}_t\cdot \nabla T_t u \,dx,
\end{equation}
which will be instrumental in order to use the monotonicity trick.{ By Theorem~\ref{theo:approx} we can take an approximating sequence $(T_t u)_{\delta} $   of smooth functions   such that  $\nabla (T_t u)_{\delta}\xrightarrow[\delta\to 0]{mod} \nabla T_t(u)$ in $L_B$ {and write}}
\begin{equation*}
\int_\Omega A(x,T_t u_k, \nabla T_{t}u_k) \nabla T_t u_k dx 
=\int_\Omega A(x,T_t u_k, \nabla T_{t}u_k) \nabla (T_t u)_\delta dx + 
\int_\Omega A(x,T_t u_k, \nabla T_{t}u_k) \big( \nabla T_t u_k - \nabla (T_t u)_\delta \big) dx
\end{equation*}
Therefore, if  we take into account Lemma~\ref{modw} and that \eqref{a-conv-ca}  holds, in order to get~\eqref{lim<A'}, it suffices to show that
\begin{equation}
\label{limsup2}
\lim_{\delta\to 0}\limsup_{k\to\infty}
\int_\Omega A(x,T_{t} u_k, \nabla T_{t} u_k) \nabla \left[T_t u_k-(T_t u)_\delta\right]dx= 0.
\end{equation}
Let us define the cut-off function $\psi_l:\R\to\R$ by 
\begin{equation}\label{psil}
\psi_l(r):= 
\min\{(l+1-|r|)^+,1\}.
\end{equation} 
Observe that since $A(x,z,0) = 0$, { $A(x,T_{t} u_k, \nabla T_{t} u_k)$ is not zero provided $|u_k|\le t$. Then for $l>t$, it is $\psi_l(u_k)=1$ and hence}
\begin{equation*}\begin{split}
\int_{\Omega} A(x,T_{t} u_k, \nabla T_{t} u_k) & \nabla \left[T_t u_k-(T_t u)_\delta\right]dx 
=
\int_\Omega A(x,T_{t} u_k, \nabla T_{t} u_k)  \nabla \left[T_t u_k-(T_t u)_\delta\right] \psi_l(u_k) \,dx  \cr\cr
 =& \int_\Omega A(x,u_k, \nabla u_k)\nabla \left[T_t u_k-(T_t u)_\delta\right]  \psi_l(u_k) \,dx \cr\cr
& - \int_\Omega \big( A(x, u_k, \nabla u_k)- A(x,T_{t} u_k, \nabla T_{t} u_k) \big) \nabla \left[T_t u_k-(T_t u)_\delta\right]  \psi_l(u_k)\,dx  \cr\cr
 =&  \int_\Omega A(x,u_k, \nabla u_k) \nabla \big( \psi_l(u_k) (T_t u_k-(T_t u)_\delta) \big) dx \cr\cr
&- \int_\Omega \big( A(x,u_k, \nabla u_k) \nabla \psi_l(u_k) \big) (T_t u_k-(T_t u)_\delta) dx\cr\cr
&- \int_\Omega \big( A(x, u_k, \nabla u_k)- A(x,T_{t} u_k, \nabla T_{t} u_k) \big)  \nabla \left[T_t u_k-(T_t u)_\delta\right]\psi_l(u_k)\,dx \cr\cr
=&I_1+I_2+I_3
\end{split}\end{equation*}
{In order to show \eqref{limsup2}, it will be enough to prove that each of the integrals in the right hand side of last equality goes to zero as $\delta \to 0$ and $k\to \infty$.}

{Note  that  $\vp=\psi_l(u_k)(T_t u_k-(T_t u)_\delta)$} is a legitimate test function for the equation~\eqref{weak:u-k} because of \eqref{1sttest}). It follows that for $I_1$ we have
\begin{equation}
\label{cf.C5}
{I_1=}\int_\Omega A(x,u_k, \nabla u_k)\nabla \big( \psi_l(u_k)(T_t u_k-(T_t u)_\delta)\big) dx=\int_\Omega f_k \psi_l(u_k)(T_t u_k-(T_t u)_\delta)dx
\end{equation}
{and, for every $l$ fixed, we have}
\[
\lim_{\delta\to 0}\limsup_{k\to\infty}I_1=\lim_{\delta\to 0}\limsup_{k\to\infty} \int_\Omega f_k\psi_l(u_k)(T_t u_k-(T_t u)_\delta)dx=0.
\]
To this end, observe that
\begin{equation*}
\begin{split}
\lim_{\delta\to 0}  \limsup_{k\to\infty}&\left|\int_\Omega f_k\psi_l(u_k)(T_t(u_k)-(T_t(u))_\delta)dx\right|\\
& \leq \lim_{\delta\to 0} \limsup_{k\to\infty}\int_\Omega |f_k-f | \,|T_t u_k-(T_t u)_\delta|dx+\lim_{\delta\to 0}\limsup_{k\to\infty}  \int_\Omega |f|\,|T_t u_k -(T_t u)_\delta|dx \\ 
&\leq \lim_{\delta\to 0} \int_\Omega |f|\,|T_t u -(T_t u)_\delta|dx=0.
\end{split}
\end{equation*}
Note that the first limit in the second line vanishes since $u_k \to u$ a.e., $f_k \to f$ in $L^1(\Omega)$ and $(T_t u)_\delta \to T_t u$ modularly (so in $L^1$). On the other hand,  the last equality holds thanks to the Lebesgue Dominated Convergence Theorem that is  legitimate to be used since $(T_t u)_\delta$ are uniformly bounded.

\noindent It follows that
\begin{equation}\label{qui}
\lim_{\delta\to 0}  \limsup_{k\to\infty}\int_\Omega A(x,u_k, \nabla u_k)\nabla \big( \psi_l(u_k)(T_t u_k-(T_t u)_\delta)\big) dx=0
\end{equation} 

\noindent To deal with $I_2$ we need to show that 
\begin{align}\label{other}
\lim_{l \to \infty} \lim_{\delta\to 0}\limsup_{k\to\infty}I_2=\lim_{l \to \infty} \lim_{\delta\to 0}\limsup_{k\to\infty} 
\int_\Omega \big( A(x,u_k, \nabla u_k) \nabla \psi_l(u_k) \big) (T_t u_k-(T_t u)_\delta) dx = 0 
\end{align}

\noindent By the definition of $\psi_l$ we first obtain that 
\begin{eqnarray} \label{II:1+2}
&&\lim_{l \to \infty} \lim_{\delta\to 0}\limsup_{k\to\infty}
\left| \int_\Omega \big( A(x,u_k, \nabla u_k) \nabla \psi_l(u_k) \big) (T_t u_k-(T_t u)_\delta) dx \right|
\leq \cr\cr
&&\lim_{l\to\infty} \lim_{\delta \to 0} \limsup_{k\to\infty}
\int_{\{l<|u_k|<l+1\}} A(x,u_k,\nabla u_k)\nabla u_k |T_t u_k-(T_t u)_\delta | dx 
\end{eqnarray}
Since $(T_tu)_{\delta}$ is uniformly bounded and for $l>t$, it is $|u_k|>t$ on the set $\{l<|u_k|<l+1\}$, the integral in the right hand side of previous inequality can be estimated by
\begin{eqnarray}
c \lim_{l\to\infty} \limsup_{k\to\infty} 
\int_{\{l<|u_k|<l+1\}} A(x,u_k, \nabla u_k)\nabla u_k dx =\!\!&c&\!\! \lim_{l\to\infty} \limsup_{k\to\infty} 
\int_\Omega A(x, u_k, \nabla u_k)(\nabla T_{l+1} u_k - \nabla T_l u_k)\, dx=\cr\cr
\!\!&c&\!\! \lim_{l\to\infty} \limsup_{k\to\infty}\int_\Omega f_k( T_{l+1} u_k -T_l u_k)\,dx
\end{eqnarray}
where we also used that $u_k$ is a solution of \eqref{weak:u-k}. Then, recalling  the pointwise inequality at   \eqref{lim-fk-to-f}, the fact that $f\in L^1$ and that ~\eqref{conv:ae:uk-to-u} holds, we obtain from previous calculations that
\begin{eqnarray}
\lim_{l \to \infty} \lim_{\delta\to 0}\limsup_{k\to\infty}I_2
\le\!\!&c&\!\! \lim_{l\to\infty} \limsup_{k\to\infty}
\int_{\{l<|u_k|<l+1\}} |f_k|\, dx \leq \cr\cr
\!\!&c&\!\! \lim_{l\to\infty} \limsup_{k\to\infty}
\int_{\{ l \leq |u_k|\}} |f_k|\, dx = 0.
\end{eqnarray}

Now we concentrate on $I_3$ and show that
\begin{eqnarray}\label{estI_3}
\lim_{l \to \infty} \lim_{\delta\to 0}\limsup_{k\to\infty}I_3=\lim_{l \to \infty} \lim_{\delta\to 0}\limsup_{k\to\infty} 
 \int_\Omega \big( A(x, u_k, \nabla u_k)- A(x,T_{t} u_k, \nabla T_{t} u_k) \big)  \nabla \left[T_t u_k-(T_t u)_\delta\right]\psi_l(u_k)\,dx=0
\end{eqnarray}
Recalling that $A(x,z,0)=0$ and the definition of the function $\psi_l$ {at} \eqref{psil}, , we have for $l>t$
\begin{align*}
\int_\Omega &\big( A(x, u_k, \nabla u_k)- A(x,T_{t} u_k, \nabla T_{t} u_k) \big)  \nabla \left[T_t u_k-(T_t u)_\delta\right]\psi_l(u_k)\,dx \\
& \quad
= \int_{\{ t < |u_k| < l+1 \}} \big( A(x, T_{l+1} u_k, \nabla T_{l+1} u_k)- A(x,T_{t} u_k, \nabla T_{t} u_k) \big)  \nabla \left[T_t u_k-(T_t u)_\delta\right]\psi_l(u_k)\,dx \\
& \quad
{= - \int_{\{ t < |u_k| < l+1 \}} A(x, T_{l+1} u_k, \nabla T_{l+1} u_k) \nabla (T_t u)_\delta \, \psi_l(u_k)  \, dx}\\
& \quad
= - \int_\Omega A(x, T_{l+1} u_k, \nabla T_{l+1} u_k) \nabla (T_t u)_\delta \, \psi_l(u_k)  \mathds{1}_{\{ t < |u_k| < l+1 \}}\, dx
 \end{align*}
 Since 
 $$
 A(x,T_t u_k ,\nabla T_{t}(u_k))\xrightharpoonup{*} \cA_t \quad  \text{weakly}-*\ \text{in}\ L_{\widetilde{B}}(\Omega) \qquad \text{as} \quad k\to \infty,
 $$
 we have
 $$
 A(x,T_{l+1} u_k ,\nabla T_{l+1}(u_k))\nabla (T_t u)_\delta \to  \cA_{l+1} \,{\nabla}(T_t u)_\delta
 \quad \text{in $L^1$}
 $$
 Therefore we have
 \begin{align*}
 \lim_{l \to \infty} &\lim_{\delta\to 0}\limsup_{k\to\infty} 
\left| \int_\Omega \big( A(x, u_k, \nabla u_k)- A(x,T_{t} u_k, \nabla T_{t} u_k) \big)  \nabla \left[T_t u_k-(T_t u)_\delta\right]\psi_l(u_k)\,dx \right|  \\
 & \quad = 
 \lim_{l \to \infty} \lim_{\delta\to 0}\limsup_{k\to\infty} 
 \left| \int_\Omega A(x, T_{l+1} u_k, \nabla T_{l+1} u_k) \nabla (T_t u)_\delta \, \psi_l(u_k)  \mathds{1}_{\{ t < |u_k| < l+1 \}}\, dx \right| \\
 & \quad = 
 \lim_{l \to \infty} \lim_{\delta\to 0}
 \left| \int_\Omega \cA_{l+1} \nabla (T_t u)_\delta \, \psi_l(u)  \mathds{1}_{\{ t < |u| < l+1 \}}\, dx \right| \\
 & \quad \leq
 \lim_{l \to \infty} 
 \int_\Omega |\cA_{l+1}|\,\, | \nabla T_t u|  \mathds{1}_{\{ t < |u| \}}\, dx \\
 &=0.
 \end{align*}
where we used Lemma \ref{lem:TM1} and Lemma~\ref{lem:ae}.

 \noindent Combining \eqref{qui}, \eqref{other} and \eqref{estI_3}, we infer that \eqref{limsup2} is true and therefore \eqref{lim<A'} holds.

\medskip

Now, using the monotonicity trick, we identify  the limit $A_t$. More precisely, now our aim is  to show that in~\eqref{a-conv-ca}
\begin{equation*}
\cA_t(x)=A(x,T_t u(x), \nabla T_t u(x))\qquad\text{a.e. in }\Omega.
\end{equation*}
Monotonicity of~$A$ results in
$$
\int_\Omega A(x,T_t u_k, \nabla T_{t} u_k ) \nabla T_t u_k \,dx
\geq 
\int_\Omega A(x, T_t u_k, \nabla T_{t} u_k ) \eta \, dx+\int_\Omega A(x,T_t u_k, \eta)  (\nabla T_t u_k-\eta)\ dx
$$
for any $\eta\in \rn$.  Taking {the} upper limit with ${k\to\infty}$ above, we have in the left hand side 
\begin{equation*}
\int_\Omega {\cal A}_t\cdot \nabla T_t u \,dx
\end{equation*} 
and for the first term in the right hand side 
$$ \int_\Omega {\cal A}_t\cdot \eta\, dx$$
respectively thanks to  \eqref{lim<A'} and  \eqref{a-conv-ca}. 

\noindent To justify that  
$$
A(x, T_t u_k, \eta) \to A(x,T_t u, \eta) \quad \text{strongly in $L_{\tilde{B}}$}
$$
we recall that $P<<B$, $A$ is continuous with respect to the second variable, and we have almost everywhere convergence of $ T_t u_k$. Altogether, we infer uniform boundedness of $\{\wt{P}(|A(x,T_t u_k,\eta)|/\lambda)\}_t$ in $L^1$ for arbitrary $\lambda>0$. Further, via Theorem~\ref{theo:delaVP}, we get uniform integrability of $\{\wt{B}(|A(x,T_t u_k,\eta)|/\lambda)\}_k$ in $L^1$ and   due to Lemma~\ref{lem:modular-norm} we get the desired limit.

\medskip

 Then, recalling that,  by \eqref{conv:basic}, we have
\begin{equation}\label{ttt} \nabla T_t u_k\xrightharpoonup[k\to\infty]{*}  \nabla T_t u\quad \text{weakly-$*$ in } L_B(\Omega),
\end{equation}
thank to Lemma~\ref{lem:weak-strong-conv} with $
A(x, T_t u_k, \eta)\in E_{\wt{B}}$ and to the continuity of  $A$,  we get  
\[\lim_{k\to\infty}\int_\Omega A(x,T_t u_k, \eta)  (\nabla T_t u_k-\eta)\, dx=\int_\Omega A(x,T_t u, \eta)  (\nabla T_t u-\eta)\, dx.\]
In conclusion, we have
\begin{equation*}
\int_\Omega {\cal A}_t\cdot \nabla T_t u \,dx\geq \int_\Omega {\cal A}_t\cdot \eta\, dx+\int_\Omega A(x,T_t u, \eta)  (\nabla T_t u-\eta)\, dx\end{equation*} 
 that it is equivalent to
\begin{equation}
\label{Ak-mon} \int_\Omega( {\cal A}_t- A(x,T_t u ,\eta) )( \nabla T_t(u)-\eta)\, dx\geq  0.
\end{equation} 
Then the monotonicity trick  (see Proposition \ref{prop:monotonicity-trick}) implies
\begin{equation} \label{lim=ca}
 A(x,T_t u , \nabla T_tu)= {\cal A}_t  \quad {\text a.e.}
\end{equation}
The convergence of the left-hand side of~\eqref{prob:trunc} follows from the facts that $u\in W^{1,B}(\Omega)$, $\nabla u$ can be understood as the generalized gradient in the sense of~\eqref{def:Zu}, and \eqref{lim=ca}, whereas the right-hand side of~\eqref{prob:trunc} converges due to~\eqref{lim-fk-to-f}. 

\medskip

\noindent{\bf Step 5. Measure data problem}

To study measure-data problems let us consider $f_k\in L^1(\Omega)\cap (W^{1,B}(\Omega))'$, given by \[f_k(x)=\int_\Omega k^n\varrho(|y-x|k)\,d\mu(y)\qquad x\in\Omega,\]
where $\varrho:\rn\to[0,\infty)$ is a standard mollifier (i.e. smooth function compactly supported in the unit ball with $\|\varrho\|_{L^1(\Omega)}=1$). Note that
\[\|f_k\|_{L^1(\Omega)}\leq 2\|\mu\|(\Omega)\]
and
\[\lim_{k\to\infty}\int_\Omega\vp\,f_k\,dx=\int_\Omega\vp\,d\mu
\]
for every $\vp\in C_c(\Omega).$ Then, for the problem~\eqref{prob:trunc} under such a choice of $f_k$ the above proof still hold.\qed

\bigskip

\subsection{Uniqueness for $L^1$-data problem with strongly monotone operator}

\bigskip

\noindent Proof of Theorem~\ref{theo:main-f}. To complete the proof of Theorem~\ref{theo:main-f} having Theorem~\ref{theo:main-mu} it suffices to infer uniqueness.  We suppose $u$ and $\bu$ are approximable solutions to problem~\eqref{eq:main:mu} with the same $L^1$-data but  which are obtained as limits of different approximate problems and prove that they have to be equal almost everywhere. By Definition~\ref{def:as:f} there exist sequences $\{f_k\}$ and $\{\bar{f_k}\}$ in $L^1(\Omega)\cap (W^{1,B}(\Omega))'$, such that $f_k\to f$ and $\bar{f_k}\to f$ in $L^1(\Omega)$ and weak solutions $u_k$ to~\eqref{prob:trunc} and $\bu_k$ to  \begin{equation}\label{prob:trunc-b}
\left\{\begin{array}{cl}
-\dv A(x,\bu_k,\nabla \bu_k)= \bar{f_k} &\qquad \mathrm{ in}\qquad  \Omega,\\
\bu_k(x)=0 &\qquad \mathrm{  on}\qquad \partial\Omega,
\end{array}\right.
\end{equation}  such that for a.e. in $\Omega$ we have both $u_k\to u$ and $\bu_k\to \bu$. We fix arbitrary $t>0$, use $\vp=T_t(u_k-\bu_k)$  as a test function in~\eqref{prob:trunc} and~\eqref{prob:trunc-b}, and subtract the equations to obtain\begin{equation}
\label{diff:u-bu}
\int_{\{|u_k-\bu_k|\leq t\}}(A(x,u_k,\nabla u_k)-A(x,\bu_k,\nabla \bu_k))\cdot( \nabla u_k-\nabla \bu_k)\,dx=\int_\Omega (f_k-\bar{f_k})T_t(u_k-\bu_k)\,dx\quad\text{for every }k\in\N.
\end{equation}
The right-hand side above tends to $0$, because $|T_t(u_k-\bu_k)|\leq t$ and for $k\to\infty$ we have $ f_k-\bar{f_k}\to 0$ in $L^1(\Omega)$. The left-hand side is convergent due to Step 4, (A3)$_s$, and Fatou's Lemma. We get
\[
\int_{\{|u -\bu |\leq t\}}(A(x,u,\nabla u )-A(x,\bu,\nabla \bu ))\cdot( \nabla u -\nabla \bu )\,dx=0.\]
Consequently,  $\nabla u =\nabla \bu$ a.e. in $\{|u -\bu |\leq t\}$ for every $t>0$, and so \begin{equation}
\label{nau=nabu}\nabla u =\nabla \bu\quad\text{ a.e. in }\Omega.
\end{equation} 
Then, using the  Poincar\'e inequality (Corollary~\ref{prop:Poincare}) with $T_r(u-T_t(\bu))$, for a fixed $r>0$, in place of $u$, we get
\[ \int_\Omega B (c_1|T_r(u-T_t(\bu))|)\,dx \leq c_2 \int_{\{|u-t|\leq r\}}B(|\nabla u|)\,dx.\]
We will prove that the left-hand side above tends to zero with $t\to\infty$. 

\noindent By using (A2) we have that  
\[ \begin{split}  \int_{\{|u-t|\leq r\}}B(|\nabla u|)\,dx &\leq {c_2}   \liminf_{k\to\infty} \int_{\{|u_k-t|\leq r\}} A(x,u_k,\nabla u_k )\nabla u_k \,dx\\
 &= {c_2}   \liminf_{k\to\infty} \int_{\{|u_k-t|\leq r\}} A(x,u_k,\nabla  u_k)\nabla T_{t+r} (u_k)  \,dx\\
 &= 2r{c_2}   \liminf_{k\to\infty} \int_{\Omega} A(x,u_k,\nabla  u_k)\nabla (1-\psi_{t,r} (u_k)) \,dx.\end{split}\]
where we introduced the notation
\[\psi_{t,r}(s)=\frac{1}{2r}\min\{1, t+r-|s|\}.\]

 Now, using weak formulation of the problem, we get
 \[ \begin{split} \int_\Omega B (c_1|T_r(u-T_t(\bu))|)\,dx &\leq c \int_{\{t-r<|u|\}} |f| \,dx\xrightarrow[t\to\infty]{}0.\end{split}\] Fatou's Lemma enables to pass to the limit  to get
\[ \int_\Omega B (c_1|T_r(u- \bu )|)\,dx=0\qquad\text{for every }r>0.\]
Therefore, $B (c_1|T_r(u- \bu )|)=0$ a.e. in $\Omega$ for every $r>0$, and consequently $u= \bu$ a.e. in $\Omega$. \qed

\section{Regularity}\label{sec:reg} 

Our next aim is to provide some regularity results  in the Orlicz-Macinkiewicz scale for the solutions of problem~\eqref{eq:main:mu} with measure data. Note that the key estimates of the proof are interesting by themselves, see Propositions~\ref{prop:grad-est-non-gr},~\ref{prop:pre-est}, and~\ref{prop:grad-est}. It is worth pointing out that we get the regularity  of the whole function $u$ and of its full gradient $\nabla u$, not only of the truncation $T_k(u)$ and its gradient.

\medskip

 The classical way of introducing the Orlicz-Marcinkiewicz spaces goes via rearrangement approach, see e.g.~\cite{Ci04,oneil}. The decreasing rearrangement $f^* : [0,\infty)\to [0, \infty]$ of a measurable function $f : \Omega\to\R$ is the unique
right-continuous, non-increasing function equidistributed with $f$, namely,
\[f^* (s) = \inf\{t \geq 0 : |\{|f|> t\}| \leq s\}\qquad\text{for }s\geq 0.\]

It's maximal rearrangement is defined as follows
\begin{equation}
\label{starstar}
f^{**}(x)=\frac{1}{x}\int_0^x f^*(t)\,dt\qquad\text{and}\qquad f^{**}(0)=f^*(0).
\end{equation}

\begin{definition}[The Orlicz--Marcinkiewicz-type spaces]\label{def:OrMar}
Let $\vp: (0, |\Omega|)\to (0,\infty)$ be a Young function. We define the Orlicz-Marcinkiewicz-type spaces 
\begin{equation}
\label{norm:OM}
{\cal M}^{\vp}(\Omega):=\left\{f\ \text{measurable in }\Omega:\   \ \ \|f\|_{{\cal M}^{\vp}(\Omega)}:=\sup_{s\in(0,|\Omega|)}\frac{f^{**}(s)}{\vp^{-1}(1/s)}<\infty\right\}
\end{equation}
and 
\begin{equation}
\label{norm:OMw}
{\cal M}_w^{\vp}(\Omega):=\left\{f\ \text{measurable in }\Omega: \ \ \|f\|_{{\cal M}_w^{\vp}(\Omega)}:=\limsup_{t\to\infty}\frac{t}{\vp^{-1}(1/|\{|f|>t\}|)}<\infty\right\}.
\end{equation} 
\end{definition} 
 While treated as a case of the Lorentz-type space the notation ${\cal M}^{\vp}(\Omega)= L^{\vp,\infty}(\Omega)$ can be also used.
 
 \begin{remark} It can be shown that $\|\cdot\|_{{\cal M}^{\vp}(\Omega)}$ defines a norm, while $\|\cdot\|_{{\cal M}_w^{\vp}(\Omega)}$ only a quasi-norm. 
 \end{remark}
Orlicz-Marcinkiewicz spaces are intermediate to Orlicz spaces in the sense that
 \[L_\vp (\Omega)\subset {\cal M}^{\vp}(\Omega) \subset {\cal M}_w^{\vp}(\Omega) \subset L_{\vp^{1-\ve}}(\Omega)\qquad\text{for all }\ {\ve\in (0,1)}.\]
Let us stress that   ${\cal M}^{\vp}(\Omega)= {\cal M}_w^{\vp}(\Omega)$ if and only if $\vp$ satisfies
\begin{equation}
\label{cond:int}\barint_0^s\vp^{-1}\left(\frac{1}{r}\right)dr\leq c  \vp^{-1}(1/s).
\end{equation}
In particular, if $p>1$ and $\beta\geq 0$, the function $\vp(t)=t^p\log^\beta(1+t)$ satisfies~\eqref{cond:int} and  hence ${\cal M}^{\vp}(\Omega)= {\cal M}_w^{\vp}(\Omega)$. 

\bigskip

We provide two types of level-sets estimates resulting from the different embeddings discussed in  Section~\ref{sec:emb}.
 
\begin{proposition} \label{prop:grad-est-non-gr}
Let $B$ be an $N$-function. Suppose $v\in {\cal T}_0^{1,B}(\Omega)$ and constants $K>0$ and $r_0\geq 0$ are  such that
\begin{equation}
\label{podpoziomice}
\int_{\{|v|<r\}}B(|\nabla v|)\,dx\leq K r\qquad\text{for }r>r_0.
\end{equation} {
Then 
\[v\in{\cal M}_w^{\Phi_1}(\Omega)\qquad\text{and}\qquad\nabla v\in{\cal M}_w^{\Psi_1}(\Omega),\]
where\begin{equation}
\label{phi1psi1}
\Phi_1(r)=\left(\frac{B(c_1 r)}{Kr}\right)^{N'}  \qquad\text{and}\qquad \Psi_1(r)= \frac{B( r)}{\overline{K}\phi^{-1}(B(r))},
\end{equation}
with $c_1=c_1({\rm diam}\,\Omega)$  and $\overline{K}=2\max\{K,K^{N'}\}$.}
\end{proposition} 

\begin{proof}
First of all we notice that since $v\in {\cal T}_0^{1,B}(\Omega)$, then of course $T_r(v) \in W_0^{1,B}(\Omega)$ for every $r>0$. Therefore by the Sobolev-Poincar\'e inequality at Proposition~\ref{prop:Sob-Poi}, we get
\[\left(\int_\Omega B^{N'}(c_1|T_r(v)|)dx\right)^\frac{1}{N'}\leq  c_2\int_{\{|v|<r\}} B ( |\nabla v|)dx.\]
To estimate the left-hand side from below we note that for every $r>0$\begin{equation}
\label{Tr}\{c_1|T_r(v)|>c_1r\}=\{|v|>r\}.
\end{equation}
Thus\begin{equation}\label{r-est} 
|\{|v|>r\}| B^{N'}(c_1r) \leq \int_{\{|v|<r\}} B^{N'} (c_1 |T_r(v) |)dx. 
\end{equation}

Summing up the above observations and taking into account~\eqref{podpoziomice} we obtain
\[|\{|v|>r\}| B^{N'}(c_1r)\leq \left(\int_{\{|v|<r\}} B (|\nabla v|)dx\right)^{N'}\leq \left(Kr\right)^{N'} \qquad \text{for }r>r_0,\]
implying
\begin{equation}\label{|u|>r}
|\{|v|>r\}| \leq \left(\frac{Kr}{B(c_1r)}\right)^{N'}\qquad\text{for }r>r_0.
\end{equation}

By using~\eqref{|u|>r} we deduce that
\[|\{B(|\nabla v|)>s \}|\leq |\{ |v|> r\}|+|\{B(|\nabla v|)>s,\, |v|\leq r\}|\leq (Kr/B(r))^{N'}+Kr/s\qquad\text{for }r>r_0,\, s>0.\]
Now recall the definition of $\phi$ and set $r=\phi^{-1}(s)$. Then, for $s> \phi (r_0)$, we get
\[|\{B(|\nabla v|)>s \}|\leq \overline{K} \frac{\phi^{-1}(s)}{s}\]
 and it suffices to take $\theta=B^{-1}(s)$ to ensure that~for $s>\phi(r_0)$
\begin{equation}\label{|gradu|>s}
 |\{ |\nabla v| >\theta \}|\leq \overline{K} \frac{\phi^{-1}(B(\theta))}{B(\theta)}.
\end{equation}
Taking into account~\eqref{|u|>r} and~\eqref{|gradu|>s} we get the claim.\qed
\end{proof}

\bigskip

Further we employ the following estimates by Cianchi and Mazy'a~\cite{CiMa}. Note that this result follows independently of the type of the growth of $B$.
\begin{proposition}[cf. Lemma 4.1, \cite{CiMa}]\label{prop:pre-est} Let $B$ be an $N$-function and $\Omega$ is a Lipschitz bounded domain. Suppose $v\in {\cal T}_0^{1,B}(\Omega)$ and  there exist constants $K>0$ and $r_0\geq 0$, such that~\eqref{podpoziomice} is satisfied.
\begin{itemize} 
\item[(a)] If~\eqref{intB}$_1$, then there exists a constant $c=c(N)$  such that  
\begin{equation}\label{|u|>r-inf}
|\{|v|>r\}|  \leq  \frac{Kr}{B_N(c r^\frac{1}{N'}/K^\frac{1}{N})}\qquad\text{for }r>r_0.
\end{equation}
\item[(b)] If~\eqref{intB}$_2$, then there exists a constant $r_1=r_1(r_0,N,M)$ such that 
\begin{equation}\label{|u|>r-fin}
|\{|v|>r\}| =0\qquad\text{for }r>r_1.
\end{equation} 
\end{itemize}
\end{proposition}
Despite~\cite[Lemma 4.1]{CiMa} is formulated assuming~\eqref{int0B}, it is explained in~\cite{CiMa} that it is not necessary. Moreover, the proof admits to consider functions from ${\cal T}_0^{1,B}(\Omega)$ not only $W_0^{1,B}(\Omega)$ as in the statement therein.

\medskip

Now we can infer gradient estimates. 
\begin{proposition} \label{prop:grad-est}
Let $B$ be an $N$-function satisfying~\eqref{int0B}  {and recall  $\phi_N$  {and $B_N$} given by~\eqref{BN}}. Suppose $v\in {\cal T}_0^{1,B}(\Omega)$ and constants $K>0$ and $r_0\geq 0$ are such that~\eqref{podpoziomice} holds. 
\begin{itemize}
\item[(a)] If~\eqref{intB}$_1$, then   $v\in{\cal M}_w^{\Phi_2}(\Omega)$ and $\nabla v\in{\cal M}_w^{\Psi_2}(\Omega)$,  {where \begin{equation}
\label{phi2psi2}
\Phi_2(r)=\frac{B_N(\bar{c}r^{(N-1)/N} )}{r}\qquad \text{and}\qquad \Psi_2 (r)=\frac{B(r)}{\phi_N(r) }
\end{equation} 
with a constant $\bar{c}=\bar{c}(N,K)$}.
\item[(b)] If~\eqref{intB}$_2$, then   {$v\in L^\infty(\Omega)$ and $\nabla v\in{\cal M}_w^{B}(\Omega)$. }
\end{itemize} \end{proposition}

\begin{proof} We notice first that~\eqref{podpoziomice} implies\[
|\{B(|\nabla v|)>s,\, |v|\leq r\}|\leq\frac{1}{s}\int_{\{B(|\nabla v|)>s,\, |v|\leq r\}} B(|\nabla v|)\,dx\leq K\frac{r}{s}\quad\text{for }r>r_0\text{ and }s> 0.\]

Let us concentrate on \textit{(a)}.  {We infer that $v\in {\cal M}_w^{\Phi_2}(\Omega)$ directly from~\eqref{|u|>r-inf}. Furthermore, since}
\begin{equation}
\label{B>s}
|\{B(|\nabla v|)>s \}|\leq |\{ |v|> r\}|+|\{B(|\nabla v|)>s,\, |v|\leq r\}|,
\end{equation} 
then by \eqref{|u|>r-inf} we get
\[|\{B(|\nabla v|)>s \}|\leq   \frac{Kr}{B_N(cr^\frac{1}{N'}/ {K}^\frac{1}{N})}+K\frac{r}{s}\qquad\text{for }r>r_0,\, s>0.\]

Substitute $r=( {K}^{1/N}B_N^{-1}(s)/c)^{N'}$ and consider $s\geq B_N(ct^{1/N'}/ {K}^{1/N})$. Then 
\[|\{B(|\nabla v|)>s \}|\leq 2\left(\frac{{K}^{\frac{1}{N}}}{c}\right)^{N'} \frac{(B_N^{-1}(s))^{N'}}{s}.\]
Taking $\theta=B^{-1}(s)$ we obtain~that there exists a constant $K_1=K_1(N,K)$  such that  
\begin{equation}\label{grad-u-inf}|\{ |\nabla v| >\theta \}|\leq K_1  \frac{(B_N^{-1}(B(\theta)))^{N'}}{B(\theta)}\quad\text{for }  \theta > 0 
\end{equation}
implying \textit{(a)}.

Now we turn to prove \textit{(b)}. { Boundedness of $v$ results directly from~\eqref{|u|>r-fin}. For estimating super-level set of its gradient  we use again~\eqref{B>s} to get}
\[|\{B(|\nabla v|)>s \}|\leq  K\frac{r}{s}\quad\text{for }r>r_2=\max\{r_0,r_1\}\text{ and }s> 0.\]
Taking $\theta=B^{-1}(s)$ we obtain
\begin{equation}\label{grad-u-fin}|\{ |\nabla v| >\theta \}|\leq  \frac{Kr_3}{B(\theta)}\quad\text{for }  \theta > 0 
\end{equation} 
for a constant $r_2=r_2(r_0,K,N)$,   
implying \textit{(b)}.\qed \end{proof}

\bigskip

Let us carry on by giving  regularity  results of approximable solutions to~\eqref{eq:main:mu} in the scale of Orlicz-Marcinkiewicz spaces. Let us mention that  results of this type were already obtained recently in the reflexive case in~\cite{CiMa,IC-grad} and in nonreflexive spaces~\cite{ACCZG,Baroni-OM}.  

\medskip

\begin{theorem}[Estimates on approximable solutions]
\label{theo:reg}
Assume  $\mu\in {\cal M}(\Omega)$  and  $A: \Omega \times \R \times \rn\to\rn$ a function satisfying assumptions (A1), (A2), (A3)$_w$ and (A4). Recall function $\Phi_1$, $\Psi_1$, and $\Phi_2$, $\Psi_2$, given by~\eqref{phi1psi1} and~\eqref{phi2psi2}, respectively. Then   every approximable solution $u\in{\cal T}^{1,B}(\Omega)$ to~\eqref{eq:main:mu} satisfies\[u\in{\cal M}_w^{\Phi_1}(\Omega)\qquad\text{and}\qquad\nabla u\in{\cal M}_w^{\Psi_1}(\Omega). \]

Moreover,
\begin{itemize}
\item[(a)] if $B$ satisfies~\eqref{intB}$_1$, then   $u\in{\cal M}_w^{\Phi_2}(\Omega)$ and $\nabla u\in{\cal M}_w^{\Psi_2}(\Omega)$;
\item[(b)] if $B$ satisfies~\eqref{intB}$_2$, then  $u\in L^\infty(\Omega)$ and $\nabla u\in{\cal M}_w^{B}(\Omega)$.
\end{itemize}
\end{theorem}  
\begin{proof} Since the approximable solutions to~\eqref{eq:main:mu} satisfy~\eqref{podpoziomice}, as a direct consequence of Propositions~\ref{prop:grad-est-non-gr}, \ref{prop:pre-est}, and~\ref{prop:grad-est}, we infer the following information on their regularity.

According to definition in~\eqref{W0} the space $W^{1,B}_0(\Omega)$ is closed in weak-* topology and in~\eqref{conv:basic} we infer that \[\nabla T_t u_k \xrightharpoonup[k\to\infty]{*}  \nabla T_t u\quad \text{weakly-$*$ in } L^B(\Omega).\] 
 \qed
\end{proof}
\smallskip

\noindent Recall that under condition~\eqref{cond:int} on $\vp$, we can substitute above each  ${\cal M}_w^{\vp}(\Omega)$ with ${\cal M}^{\vp}(\Omega)$, cf. Definition~\ref{def:OrMar}.

\bigskip

We give examples related to the case of the Zygmund-type modular functions and extending this setting. 
 
\begin{example}[Zygmund-type functions] Consider $B(t)\sim t^p\log^\beta(1+t)$  with $p>1$ and $\beta\geq 0$ near infinity. Then $B \in \Delta_2$. Our framework admitts to use~\cite[Lemma 4.5]{CiMa} to get estimates for  {approximable solutions} to~\eqref{eq:main:mu} with {$L^1$--data} as in~\cite[Example 3.4]{CiMa}, in particular implying what follows.
\begin{itemize}
\item[(slow)] If $1<p<n$, then $u\in {\cal M}^{\Phi}(\Omega)=L^{\frac{n(p-1)}{n-p},\infty}(\log L)^\frac{\beta p}{n-p}(\Omega)$ and $\nabla u\in {\cal M}^{\Psi}(\Omega)=L^{\frac{(p-1)n}{n-1},\infty}(\log L)^\frac{\beta n}{n-1}(\Omega).$ 
\item[(fast)] If $p>n$, or $p=n$ and $\beta>n-1$, then $u\in L^\infty$ and $\nabla u\in {\cal M}^B(\Omega)=L^{n,\infty}(\log L)^\beta(\Omega).$
\end{itemize}
Let us point out that there is a  misprint in powers in~\cite[Example 3.4]{CiMa}.
\end{example}

\begin{example}[Outside $\Delta_2$ or polynomial control] We have the following examples.
\begin{itemize}
\item[(slow)] If $B(t)=t\log(1+ t)\in \Delta_2$, but is not controlled by two power functions greater than $1$. Indeed, $B(t)\not\geq t^{1+\ve}$ for any $\ve>0$. Then $u\in {\cal M}^{t\log^{N'}(1+t)}(\Omega)$ and $\nabla u\in{L\log L}(\Omega)$.
\item[(fast)]  If $B(t)=t\exp t\not\in \Delta_2$, it grows faster than any power and then $u\in L^\infty$ and $\nabla u\in {\cal M}_w^{t\exp t}(\Omega).$
\item[(***)] We can also infer estimates of the Orlicz-Marcinkiewicz--type, when the modular function is irregular:  trapped between two power--type functions, but does not satisfy $\Delta_2$-condition, see Example~\ref{ex:poly-non-D2} or~\cite{BDMS}.
\end{itemize}
\end{example}

\section*{References} 
\def\ocirc#1{\ifmmode\setbox0=\hbox{$#1$}\dimen0=\ht0 \advance\dimen0
  by1pt\rlap{\hbox to\wd0{\hss\raise\dimen0
  \hbox{\hskip.2em$\scriptscriptstyle\circ$}\hss}}#1\else {\accent"17 #1}\fi}
  \def\cprime{$'$} \def\ocirc#1{\ifmmode\setbox0=\hbox{$#1$}\dimen0=\ht0
  \advance\dimen0 by1pt\rlap{\hbox to\wd0{\hss\raise\dimen0
  \hbox{\hskip.2em$\scriptscriptstyle\circ$}\hss}}#1\else {\accent"17 #1}\fi}
  \def\ocirc#1{\ifmmode\setbox0=\hbox{$#1$}\dimen0=\ht0 \advance\dimen0
  by1pt\rlap{\hbox to\wd0{\hss\raise\dimen0
  \hbox{\hskip.2em$\scriptscriptstyle\circ$}\hss}}#1\else {\accent"17 #1}\fi}
  \def\cprime{$'$}

\end{document}